\documentclass{gOMS2e}

\usepackage{epstopdf}
\usepackage{subfigure}
\usepackage[framemethod=tikz]{mdframed}

\def\dom{{\rm dom \,}}
\def\beq{\begin{equation}}
\def\eeq{\end{equation}}

\def\E{\mathbb{E}}

\def\R{\mathbb{R}}

\newtheorem{assumption}{Assumption}
\newtheorem{example}{Example}
\def\ba{\begin{array}}
\def\ea{\end{array}}
\def\beann{\begin{eqnarray*}}
\def\eeann{\end{eqnarray*}}
\def\bea{\begin{eqnarray}}
\def\eea{\end{eqnarray}}

\def\BT{\begin{theorem}}
\def\ET{\end{theorem}}
\def\BL{\begin{lemma}}
\def\EL{\end{lemma}}
\def\BC{\begin{corollary}}
\def\EC{\end{corollary}}
\def\BE{\begin{example}}
\def\EE{\end{example}}
\def\BD{\begin{definition}}
\def\ED{\end{definition}}
\def\BR{\begin{remark}}
\def\ER{\end{remark}}
\def\BAS{\begin{assumption}}
\def\EAS{\end{assumption}}
\def\BI{\begin{itemize}}
\def\EI{\end{itemize}}

\def\BMP{\begin{minipage}{9.5cm}}
\def\EMP{\end{minipage}}
\def\MPT{\begin{minipage}{11.5cm}}
\def\EPT{\end{minipage}}

\def\la{\langle}
\def\ra{\rangle}

\theoremstyle{plain}
\newtheorem{theorem}{Theorem}[section]
\newtheorem{corollary}[theorem]{Corollary}
\newtheorem{lemma}[theorem]{Lemma}

\theoremstyle{definition}
\newtheorem{definition}{Definition}

\theoremstyle{remark}
\newtheorem{remark}{Remark}

\begin{document}

\jvol{36} \jnum{1}\jyear{2021}


\title{On Inexact Solution of Auxiliary Problems in Tensor Methods for Convex Optimization}

\author{
\name{G.N. Grapiglia\textsuperscript{a}$^{\ast}$\thanks{$^\ast$Corresponding author. Email: grapiglia@ufpr.br}, Yu. Nesterov\textsuperscript{b}}
\affil{\textsuperscript{a}Departamento de Matem\'atica, Universidade Federal do Paran\'a, Centro Polit\'ecnico, Cx. postal 19.081, 81531-980, Curitiba, Paran\'a, Brazil;\\
\textsuperscript{b}Center for Operations Research and Econometrics (CORE), Catholic University of Louvain
(UCL), 34 voie du Roman Pays, 1348 Louvain-la-Neuve, Belgium}
\received{November 22, 2019}
}

\maketitle

\begin{abstract}
In this paper we study the auxiliary problems that appear in $p$-order tensor methods for unconstrained minimization of convex functions with $\nu$-H\"{o}lder continuous $p$th derivatives. This type of auxiliary problems corresponds to the minimization of a $(p+\nu)$-order regularization of the $p$th order Taylor approximation of the objective. For the case $p=3$, we consider the use of Gradient Methods with Bregman distance. When the regularization parameter is sufficiently large, we prove that the referred methods take at most $\mathcal{O}(\log(\epsilon^{-1}))$ iterations to find either a suitable approximate stationary point of the tensor model or an $\epsilon$-approximate stationary point of the original objective function.
\end{abstract}

\begin{keywords}
unconstrained minimization; high-order methods; tensor methods; H\"older condition; worst-case global complexity
bounds
\end{keywords}

\begin{classcode}49M15; 49M37; 58C15; 90C25; 90C30
\end{classcode}

\section{Introduction}
\setcounter{equation}{0}

\subsection{Motivation}

In \cite{NP}, a cubic regularization of Newton's method (CNM) was proposed for convex and nonconvex minimization of functions with Lipschitz continuous Hessian. At each iteration of CNM a trial point is computed by minimizing a third-order regularization of the second-order Taylor approximation of the objective function around the current iterate. When the objective $f$ is convex, it was shown that CNM takes at most $\mathcal{O}(\epsilon^{-1/2})$ iterations to generate $\bar{x}$ such that $f(\bar{x})-f_{*}\leq\epsilon$, where $f_{*}$ is the optimal value of $f$. An accelerated version of CNM was proposed in \cite{NES3} with an improved complexity bound of $\mathcal{O}(\epsilon^{-1/3})$. In the sequel, accelerated $p$-order tensor methods with complexity of $\mathcal{O}(\epsilon^{-1/(p+1)})$ were proposed by Baes \cite{BAES}, generalizing the accelerated CNM. However, each iteration of these tensor methods require the exact minimization of a potentially nonconvex model, namely, a $(p+1)$-order regularization of the $p$th order Taylor approximation of the objective. Since the global minimization of general nonconvex multivariate polynomials is computationally out of reach, the contribution in \cite{BAES} remained restricted to the theoretical field.

Recently, two important works have pointed new ways towards practical tensor methods. In the context of nonconvex optimization, Birgin et al. \cite{Birgin} presented a $p$-order tensor method that can find $\bar{x}$ with $\|\nabla f(\bar{x})\|_{*}\leq\epsilon$  in at most $\mathcal{O}(\epsilon^{-\frac{p+1}{p}})$ iterations, generalizing the bound of $\mathcal{O}(\epsilon^{-\frac{3}{2}})$ proved in \cite{NP} for the CNM (case $p=2$). The method is based on the same regularized models used in \cite{BAES}, but allows the trial points to be approximate stationary points of the tensor models. On the other hand, in the context of convex optimization, Nesterov \cite{NES6} proved that regularized tensor models are convex if the corresponding regularization parameter is sufficiently large. This makes possible the iterative solution of tensor auxiliary problems by efficient methods from Convex Optimization. 

The tensor methods in \cite{NES6} make explicit use of the Lipschitz constant of the higher-order derivative of the objective and also require the exact solution of the convex auxiliary problems. In \cite{GN3,GN4}, we proposed adaptive tensor methods for unconstrained minimization of convex functions with $\nu$-H\"{o}lder continuous $p$th derivatives. These methods generalize the regularized Newton methods presented in \cite{GN,GN2} for $p=2$, and allow inexact solution of the auxiliary problems as in \cite{Birgin}. 

In this paper, we investigate the use of Gradient Methods with Bregman distance to approximately solve the auxiliary problems in third-order tensor methods. When the regularization parameter is sufficiently large, we prove that these schemes applied to the corresponding tensor model take at most $\mathcal{O}(\log(\epsilon^{-1}))$ iterations to find either an approximate stationary point of the model  (in the sense of \cite{Birgin}) or an $\epsilon$-approximate stationary point of the original objective function.

\subsection{Contents}
The paper is organized as follows. In section 2, we state the general problem. In section 3, we establish convexity and smoothness properties of regularized third-order tensor models. In Section 4, we consider a Bregman Gradient Method for the approximate solution of smooth third-order tensor auxiliary problems. In section 4, we consider possibly nonsmooth auxiliary problem that arise in composite convex optimization. General complexity results for Bregman Gradient Methods are provided in the Appendix.

\subsection{Notations and Generalities}

In what follows, we denote by $\E$ a finite-dimensional
real vector space, and by $\E^*$ its {\em dual} space,
composed by linear functionals on $\E$. The value of
function $s \in \E^*$ at point $x \in \E$ is denoted by
$\la s, x \ra$. Given a self-adjoint positive definite 
operator $B:\E \to \E^*$ (notation $B \succ 0$), we can
endow these spaces with conjugate Euclidean norms:
$$
\ba{rcl}
\| x \| & = & \la B x, x \ra^{1/2},\; x\in \E, \quad \| s
\|_* \; = \; \la s, B^{-1} s \ra^{1/2}, \; s \in \E^*.
\ea
$$
For a smooth function $f:\text{dom}\,f\to \R$ with convex and
open domain $\text{dom}\,f\subset\E$, denote by $\nabla f(x)$
its gradient, and by $\nabla^{2}f(x)$ its Hessian evaluated at 
point $x\in\text{dom}\,f$. Note that $\nabla f(x)\in\E^{*}$ and 
$\nabla^{2}f(x)h\in\E^{*}$ for $x\in\text{dom}\,f$ and $h\in\E$.

For any integer $p\geq 1$, denote by
\begin{equation*}
D^{p}f(x)[h_{1},\ldots,h_{p}]
\end{equation*} 
the directional derivative of function $f$ at $x$ along directions
$h_{i}\in\E$, $i=1,\ldots,p$. In particular, for any $x\in\text{dom}\,f$ and $h_{1},h_{2}\in\E$ we have
\begin{equation*}
Df(x)[h_{1}]=\langle\nabla f(x),h_{1}\rangle\quad\text{and}\quad D^{2}f(x)[h_{1},h_{2}]=\langle\nabla^{2}f(x)h_{1},h_{2}\rangle.
\end{equation*}
If $h_{1}=\ldots=h_{p}=h\in\E$, we denote $D^{p}f(x)[h_{1},\ldots,h_{p}]$ as $D^{p}f(x)[h]^{p}$. With this notation, the $p$th order Taylor approximation of function $f$ at $x\in\text{dom}\,f$ can be written as follows:
\begin{equation}
f(x+h)=\Phi_{x,p}(x+h)+o(\|h\|^{p}),\,\,x+h\in\text{dom}\,f,
\label{eq:1.1}
\end{equation}
where
\begin{equation}
\Phi_{x,p}(y)\equiv f(x)+\sum_{i=1}^{p}\dfrac{1}{i!}D^{i}f(x)[y-x]^{i},\,\,y\in\E.
\label{eq:1.2}
\end{equation}
Since $D^{p}f(x)[\,.\,]$ is a 
symmetric $p$-linear form, its norm is defined as:
\begin{equation*}
\|D^{p}f(x)\|=\max_{h_{1},\ldots,h_{p}}\left\{\left|D^{p}f(x)[h_{1},\ldots,h_{p}]\right|\,:\,\|h_{i}\|\leq 1,\,i=1,\ldots,p\right\}.
\end{equation*}
It can be shown that (see, e.g., Appendix 1 in \cite{NES7})
\begin{equation*}
\|D^{p}f(x)\|=\max_{h}\left\{\left|D^{p}f(x)[h]^{p}\right|\,:\,\|h\|\leq 1\right\}.
\end{equation*}
Similarly, since $D^{p}f(x)[.\,,\ldots,\,.]-D^{p}f(y)[.,\ldots,.]$ is also a symmetric $p$-linear form for fixed $x,y\in\text{dom}\,f$, it follows that
 \begin{equation*}
\|D^{p}f(x)-D^{p}f(y)\|=\max_{h}\left\{\left|D^{p}f(x)[h]^{p}-D^{p}f(y)[h]^{p}\right|\,:\,\|h\|\leq 1\right\}.
\end{equation*}

\section{Problem Statement}

Let $f:\E\to\mathbb{R}$ be a $p$-times differentiable convex function with $\nu$-H\"{o}lder continuous $p$th derivatives, that is,
\begin{equation}
\|D^{p}f(x)-D^{p}f(y)\|\leq H_{f,p}(\nu)\|x-y\|^{\nu},\quad\forall x,y\in\E,
\label{eq:2.1}
\end{equation}
for some $\nu\in [0,1]$. Given $x\in\E$, let us consider the following minimization problem:
\begin{equation}
\min_{y\in\E}\Omega_{x,p,H}^{(\nu)}(y)\equiv \Phi_{x,p}(y)+\dfrac{H}{p!}\|y-x\|^{p+\nu},
\label{eq:2.2}
\end{equation}
where $\Phi_{x,p}(\,.\,)$ is defined in (\ref{eq:1.2}) and $H>0$.
Problems of the form (\ref{eq:2.2}) appear as auxiliary problems in $p$-order tensor methods for convex and nonconvex unconstrained optimization (see, e.g., \cite{Birgin,Mart,CGT2,GN3,GN4}). In these methods, only approximate stationary points of $\Omega_{x,p,H}^{(\nu)}(\,.\,)$ are required \cite{Birgin}. Specifically, it is enough to find $x^{+}$ such that
\begin{equation}
\Omega_{x,p,H}^{(\nu)}(x^{+})\leq f(x),
\label{eq:2.4}
\end{equation}
and
\begin{equation}
\|\nabla\Omega_{x,p,H}^{(\nu)}(x^{+})\|_{*}\leq\theta\|x^{+}-x\|^{p+\nu-1},
\label{eq:2.5}
\end{equation}
where $\theta>0$. The next lemma gives a sufficient condition for (\ref{eq:2.5}) to be satisfied.

\begin{lemma}
\label{lem:2.1}
Let $x\in\E$, $H,\theta>0$ and $\delta\in (0,1)$. If 
\begin{equation}
\|\nabla f(x^{+})\|_{*}\geq\delta\quad\text{and}\quad \|\nabla\Omega_{x,p,H}^{(\nu)}(x^{+})\|_{*}\leq\min\left\{\frac{1}{2},\frac{\theta (p-1)!}{2[H_{f,p}(\nu)+H(p+\nu)]}\right\}\delta,
\label{eq:2.6}
\end{equation}
then $x^{+}$ satisfies (\ref{eq:2.5}).
\end{lemma}
\begin{proof}
From (\ref{eq:2.1}), it follows that
\begin{equation}
\|\nabla f(y)-\nabla\Phi_{x,p}(y)\|_{*}\leq\dfrac{H_{f,p}(\nu)}{(p-1)!}\|y-x\|^{p+\nu-1},\quad\forall y\in\E.
\label{eq:2.7}
\end{equation}
Combining (\ref{eq:2.6}) and (\ref{eq:2.7}) we obtain
\begin{eqnarray*}
\delta\leq\|\nabla f(x^{+})\|_{*}&\leq &\|\nabla f(x^{+})-\nabla\Phi_{x,p}(x^{+})\|_{*}+\|\nabla\Phi_{x,p}(x^{+})-\nabla\Omega_{x,p,H}^{(\nu)}(x^{+})\|_{*}\\
& & +\|\nabla\Omega_{x,p,H}^{(\nu)}(x^{+})\|_{*}\\
&\leq & \dfrac{H_{f,p}(\nu)}{(p-1)!}\|x^{+}-x\|^{p+\nu-1}+\dfrac{H(p+\nu)}{p!}\|x^{+}-x\|^{p+\nu-1}+\dfrac{\delta}{2}\\
&\leq & \dfrac{H_{f,p}(\nu)+H(p+\nu)}{(p-1)!}\|x^{+}-x\|^{p+\nu-1}+\dfrac{\delta}{2}.
\end{eqnarray*}
Thus,
\begin{equation*}
\dfrac{\delta}{2}\leq\left(\dfrac{H_{f,p}+H(p+\nu)}{(p-1)!}\right)\|x^{+}-x\|^{p+\nu-1},
\end{equation*}
which gives
\begin{equation}
\left[\dfrac{\theta (p-1)!}{2[H_{f,p}(\nu)+H(p+\nu)]}\right]\delta\leq\theta\|x^{+}-x\|^{p+\nu-1}.
\label{eq:2.8}
\end{equation}
Finally, (\ref{eq:2.5}) follows directly from the second inequality in (\ref{eq:2.6}) and (\ref{eq:2.8}).
\end{proof}

In view of Lemma \ref{lem:2.1}, $x^{+}$ satisfying (\ref{eq:2.4})-(\ref{eq:2.5}) can be computed by any monotone optimization scheme that drives the gradient of the objective to zero. It is worth mentioning that the lemma above does not require the convexity of $f$. Therefore, a slight modification of it also applies to the tensor models in \cite{Birgin,Mart,CGT2}. Our goal in the next sections is to describe iterative schemes to solve (\ref{eq:2.2}) with $p=3$, and also provide iteration complexity bounds for reducing the norm of the gradient below the threshold specified in the second inequality in (\ref{eq:2.6}).

\section{Gradient Method for Smooth Third-Order Tensor Models}

\subsection{Convexity and Relative Smoothness Properties}

The next lemma gives a sufficient condition for function $\Omega_{x,p,H}^{(\nu)}(\,.\,)$ to be convex.
\begin{lemma}
\label{lem:3.1}
Let $p\geq 2$. Then, for any $x,y\in\E$ we have
\begin{equation}
\nabla^{2}f(y)\preceq \nabla^{2}\Phi_{x,p}(y)+\dfrac{H_{f,p}(\nu)}{(p-2)!}\|y-x\|^{p+\nu-2}B. 
\label{eq:3.3}
\end{equation}
Moreover, if $H\geq (p-1)H_{f,p}(\nu)$, then function $\Omega_{x,p,H}^{(\nu)}(\,.\,)$ is convex for any $x\in\E$.
\end{lemma}

\begin{proof}
See Lemma 5.1 in \cite{GN4}.
\end{proof}

In order to exploit additional properties of $\Omega_{x,p,H}^{(\nu)}(\,.\,)$, let us focus on the case $p=3$. Note that
\begin{equation}
\Phi_{x,3}(y)=f(x)+\langle\nabla f(x),y-x\rangle +\dfrac{1}{2}\langle\nabla^{2}f(x)(y-x),(y-x)\rangle +\dfrac{1}{6}D^{3}f(x)[y-x]^{3},
\label{eq:3.1}
\end{equation}
and
\begin{equation}
\Omega_{x,3,H}^{(\nu)}(y)=\Phi_{x,3}(y)+\dfrac{H}{6}\|y-x\|^{3+\nu}.
\label{eq:3.2}
\end{equation}

The next auxiliary result gives bounds on the third-order derivatives of $f$. Its proof is an adaptation of the proof of Lemma 3 in \cite{NES6}.
\begin{lemma}
\label{lem:3.2}
For any $x,y\in\E$ and $\tau>0$ we have
\begin{equation}
-\dfrac{1}{\tau}\nabla^{2}f(x)-\tau^{\nu}H_{f,3}(\nu)\|y-x\|^{1+\nu}B\preceq D^{3}f(x)[y-x]\preceq \dfrac{1}{\tau}\nabla^{2}f(x)+\tau^{\nu}H_{f,3}(\nu)\|y-x\|^{1+\nu}B.
\label{eq:3.4}
\end{equation}
\end{lemma}

\begin{proof}
Given $u,y\in\E$, by Lemma \ref{lem:3.1} (for $p=3$) and the convexity of $f$, we have:
\begin{eqnarray*}
0&\leq &\la\nabla^{2}f(y)u,u\ra\leq\la\nabla^{2}\Phi_{x,3}(y)u,u\ra + H_{f,3}(\nu)\|y-x\|^{1+\nu}\|u\|^{2}\\
 & =   &\la\left(\nabla^{2}f(x)+D^{3}f(x)[y-x]\right)u,u\ra +H_{f,3}(\nu)\|y-x\|^{1+\nu}\|u\|^{2}.
\end{eqnarray*}
Thus, replacing $y$ by $\bar{y}=x+\tau (y-x)$, we obtain
\begin{eqnarray*}
0&\leq &\la\nabla^{2}f(\bar{y})u,u\ra\leq\la\nabla^{2}f(x)u,u\ra+\tau\la D^{3}f(x)[y-x]u,u\ra+\tau^{1+\nu}H_{f,3}(\nu)\|y-x\|^{1+\nu}\|u\|^{2}
\end{eqnarray*}
\begin{equation*}
\Longrightarrow -\tau\la D^{3}f(x)[y-x]u,u\ra\leq \la\nabla^{2}f(x)u,u\ra +\tau^{1+\nu}H_{f,3}(\nu)\|y-x\|^{1+\nu}\|u\|^{2}.
\end{equation*}
Then, dividing this inequality by $-\tau$, it follows that
\begin{equation}
\la D^{3}f(x)[y-x]u,u\ra\geq -\dfrac{1}{\tau}\la\nabla^{2}f(x)u,u\ra--\tau^{\nu}H_{f,3}(\nu)\|y-x\|^{1+\nu}\|u\|^{2}.
\label{eq:3.5}
\end{equation}
Since $u$ is arbitrary, this gives the first inequality in (\ref{eq:3.4}). The second inequality in (\ref{eq:3.4}) can be obtained by replacing $y-x$ by $-(y-x)$ in (\ref{eq:3.5}).
\end{proof}

Now, using Lemma \ref{lem:3.2}, we can prove relative smoothness properties\footnote{See also \cite{BBT} for a version of relative smoothness without strong convexity.} \cite{LU} of $\Omega_{x,3,H}^{(\nu)}(\,.\,)$.
\begin{theorem}
Let $\tau_{H}=\left[\dfrac{(3+\nu)H}{6H_{f,3}(\nu)}\right]^{\frac{1}{1+\nu}}$ and
\begin{equation}
\rho_{x}(y)\equiv \dfrac{1}{2}\langle\nabla^{2}f(x)(y-x),y-x\rangle + \dfrac{1}{3+\nu}\|y-x\|^{3+\nu}.
\label{eq:3.6}
\end{equation}
Then, the following assertions hold:
\begin{itemize}
\item[(a)] Function $\Omega_{x,3,H}^{(\nu)}(\,.\,)$ is $L_{H}$-smooth relative to $\rho_{x}(\,.\,)$ for 
\begin{equation}
L_{H}=\max\left\{\dfrac{\tau_{H}+1}{\tau_{H}},\tau_{H}^{\nu}(\tau_{H}+1)H_{f,3}(\nu)\right\}.
\label{eq:3.7}
\end{equation}
\item[(b)] If $\tau_{H}\geq 1$, then function $\Omega_{x,3,H}^{(\nu)}(\,.\,)$ is $\mu_{H}$-strongly convex relative to $\rho_{x}(\,.\,)$ for 
\begin{equation}
\mu_{H}=\min\left\{\dfrac{\tau_{H}-1}{\tau_{H}},\tau_{H}^{\nu}(\tau_{H}-1)H_{f,3}(\nu)\right\}.
\label{eq:3.8}
\end{equation}
\end{itemize}
\label{thm:3.3}
\end{theorem}

\begin{proof}
In view of (\ref{eq:3.2}) and (\ref{eq:3.4}), we have
\begin{eqnarray*}
\nabla^{2}\Omega_{x,3,H}^{(\nu)}(y)&=&\nabla^{2}f(x)+D^{3}f(x)[y-x]+\nabla^{2}\left(\dfrac{H}{6}\|y-x\|^{3+\nu}\right)\nonumber\\
                                   &\preceq & \left(\dfrac{\tau_{H}+1}{\tau_{H}}\right)\nabla^{2}f(x)+\tau_{H}^{\nu}H_{f,3}(\nu)\|y-x\|^{1+\nu}B+\nabla^{2}\left(\dfrac{H}{6}\|y-x\|^{3+\nu}\right)\nonumber\\
                                   &\preceq & \left(\dfrac{\tau_{H}+1}{\tau_{H}}\right)\nabla^{2}f(x)+\tau_{H}^{\nu}H_{f,3}(\nu)\nabla^{2}\left(\dfrac{1}{3+\nu}\|y-x\|^{3+\nu}\right)\\
                                   & & +\dfrac{(3+\nu)H}{6}\nabla^{2}\left(\dfrac{1}{3+\nu}\|y-x\|^{3+\nu}\right)\nonumber\\
                                   & = & \left(\dfrac{\tau_{H}+1}{\tau_{H}}\right)\nabla^{2}f(x)+\tau_{H}^{\nu}(\tau_{H}+1)H_{f,3}(\nu)\nabla^{2}\left(\dfrac{1}{3+\nu}\|y-x\|^{3+\nu}\right)\nonumber\\
                                   &\preceq & \max\left\{\dfrac{\tau_{H}+1}{\tau_{H}},\tau_{H}^{\nu}(\tau_{H}+1)H_{f,3}(\nu)\right\}\left[\nabla^{2}f(x)+\nabla^{2}\left(\dfrac{1}{3+\nu}\|y-x\|^{3+\nu}\right)\right]\nonumber\\
                                   &=& L_{H}\nabla^{2}\rho_{x}(y).
\end{eqnarray*}

Since $\rho_{x}(\,.\,)$ is convex, by Proposition 1.1 in \cite{LU} we conclude that $\Omega_{x,3,H}^{(\nu)}(\,.\,)$ is $L_{H}$-smooth relative to $\rho_{x}(\,.\,)$. This proves (a).

Now, suppose that $\tau_{H}\geq 1$. In this case, by (\ref{eq:3.2}) and (\ref{eq:3.4}) we have
\begin{eqnarray*}
\nabla^{2}\Omega_{x,3,H}^{(\nu)}(y)&=&\nabla^{2}f(x)+D^{3}f(x)[y-x]+\nabla^{2}\left(\dfrac{H}{6}\|y-x\|^{3+\nu}\right)\\
 &\succeq & \left(\dfrac{\tau_{H}-1}{\tau_{H}}\right)\nabla^{2}f(x)-\tau_{H}^{\nu}H_{f,3}(\nu)\|y-x\|^{1+\nu}B+\dfrac{(3+\nu)H}{6}\nabla^{2}\left(\dfrac{1}{3+\nu}\|y-x\|^{3+\nu}\right)\\
 &\succeq & \left(\dfrac{\tau_{H}-1}{\tau_{H}}\right)\nabla^{2}f(x)-\tau_{H}^{\nu}H_{f,3}(\nu)\nabla^{2}\left(\dfrac{1}{3+\nu}\|y-x\|^{3+\nu}\right)\\
 & &+\dfrac{(3+\nu)H}{6}\left(\dfrac{1}{3+\nu}\|y-x\|^{3+\nu}\right)\\
 & = & \left(\dfrac{\tau_{H}-1}{\tau_{H}}\right)\nabla^{2}f(x)+\tau_{H}^{\nu}(\tau_{H}-1)H_{f,3}(\nu)\nabla^{2}\left(\dfrac{1}{3+\nu}\|y-x\|^{3+\nu}\right)\\
 &\succeq &\min\left\{\dfrac{\tau_{H}-1}{\tau_{H}},\tau_{H}^{\nu}(\tau_{H}-1)H_{f,3}(\nu)\right\}\left[\nabla^{2}f(x)+\nabla^{2}\left(\dfrac{1}{3+\nu}\|y-x\|^{3+\nu}\right)\right]\\
 & = & \mu_{H}\nabla^{2}\rho_{x}(y).
\end{eqnarray*}
Thus, by Proposition 1.1 in \cite{LU}, we conclude that $\Omega_{x,3,H}^{(\nu)}(\,.\,)$ is $\mu_{H}$-strongly convex relative to $\rho_{x}(\,.\,)$, and this proves (b).
\end{proof}

\begin{remark}
\label{rem:3.1}
Note that
\begin{equation*}
\nabla^{2}\left(\frac{1}{3+\nu}\|y-x\|^{3+\nu}\right)=(1+\nu)\|y-x\|^{\nu-1}B(y-x)(y-x)^{T}B+\|y-x\|^{1+\nu}B.
\end{equation*}
Consequently, for all $y\in\E$, we have
\begin{equation}
\|\nabla^{2}\rho_{x}(y)\|\leq \|\nabla^{2}f(x)\|+(2+\nu)\|y-x\|^{1+\nu},
\label{eq:3.9}
\end{equation}
where $\|A\|=\max_{\|h\|=1}\,\|Ah\|$, for any matrix $A$.
Moreover, by Lemma 5 in \cite{DOI}, it follows that $\rho_{x}(\,.\,)$ is uniformly convex of degree $3+\nu$ with parameter $2^{-(1+\nu)}$.
\end{remark}

The next lemma establishes an upper bound for the Hessians of function $\rho_{x}(\,.\,)$ when $H\geq H_{f,p}(\nu)$. 
\begin{lemma}
Given $x\in\E$ and $H\geq H_{f,3}(\nu)$, let 
\begin{equation*}
\mathcal{L}_{H}(x)=\left\{z\in\E\,:\,\Omega_{x,3,H}^{(\nu)}(z)\leq f(x)\right\}.
\end{equation*}
Suppose that $f$ has a global minimizer $x^{*}$ and that 
\begin{equation*}
x\in\mathcal{F}(x_{0})\equiv \left\{z\in\E\,:\,f(z)\leq f(x_{0})\right\},
\end{equation*}
with
\begin{equation}
\sup_{y\in\mathcal{F}(x_{0})}\|y-x^{*}\|\leq R_{0}<+\infty,
\label{eq:3.10}
\end{equation}
and $R_{0}\geq 1$. Then,
\begin{equation}
\sup\left\{\|\nabla^{2}\rho_{x}(y)\|\,:\,y\in\text{co}\left(\mathcal{L}_{H}(x)\right)\right\}\leq \|\nabla^{2}f(x)\|+12R_{0}^{2}\equiv N_{x},
\label{eq:3.11}
\end{equation}
where $\text{co}\left(X\right)$ denotes the convex hull of the set $X$.
\label{lem:3.5}
\end{lemma}

\begin{proof}
If $y\in\mathcal{L}_{H}(x)$, then
\begin{equation}
\Omega_{x,3,H}^{(\nu)}(y)\leq f(x)\leq f(x_{0}).
\label{eq:3.12}
\end{equation}
Since $H\geq H_{f,3}(\nu)$, it follows from (\ref{eq:2.1}) that
\begin{equation}
f(y)\leq\Omega_{x,3,H}^{(\nu)}(y).
\label{eq:3.13}
\end{equation}
Combining (\ref{eq:3.12}) and (\ref{eq:3.13}), we conclude that $y\in\mathcal{F}(x_{0})$ and, by (\ref{eq:3.10}), we obtain
\begin{equation}
\|y-x\|\leq \|y-x^{*}\|+\|x^{*}-x\|\leq 2R_{0}.
\label{eq:3.14}
\end{equation}
Now, let $y\in\text{co}\left(\mathcal{L}_{H}(x)\right)$. Then, there exists $\lambda\in [0,1]$ and $y_{1},y_{2}\in\mathcal{L}_{H}(x)$ such that $y=(1-\lambda)y_{1}+\lambda y_{2}$. Consequently, using (\ref{eq:3.14}), we get
\begin{equation}
\|y-x\|\leq (1-\lambda)\|y_{1}-x\|+\lambda\|y_{2}-x\|\leq 2R_{0}.
\label{eq:3.15}
\end{equation} 
Finally, by (\ref{eq:3.9}) and (\ref{eq:3.15}), we conclude that (\ref{eq:3.11}) holds.
\end{proof}

Even when $H<H_{f,p}(\nu)$ and $x\notin\mathcal{F}(x_{0})$, we can bound the Hessians of $\rho_{x}(\,.\,)$ on $\text{co}\left(\mathcal{L}_{H}(x)\right)$. For that, we need first to establish the coercivity of $\Omega_{x,3,H}^{(\nu)}(\,.\,)$ when $\nu\neq 0$.

\begin{lemma}
\label{lem:mais3.1}
Let $x\in\E$, $H>0$ and $\nu\neq 0$. Then, the following statements are true:
\begin{itemize}
\item[(a)] Given $A>0$, if
\begin{equation}
\|y-x\|>\max\left\{\left[\frac{6(A-f(x))}{H}\right]^{\frac{1}{3}},\left[\frac{6\|\nabla f(x)\|_{*}}{H}\right]^{\frac{1}{2}},\frac{3\|\nabla^{2}f(x)\|}{H},\left[3+\frac{\|D^{3}f(x)\|}{H}\right]^{\frac{1}{\nu}}\right\},
\label{eq:mais3.1}
\end{equation}
then $\Omega_{x,3,H}^{(\nu)}(y)>A$.
\item[(b)] $\Omega_{x,3,H}^{(\nu)}(\,.\,)$ is coercive.
\end{itemize}
\end{lemma}

\begin{proof}
First, by the definition of $\Omega_{x,3,H}(\,.\,)$ and the Cauchy-Schwarz inequality, we obtain
\begin{eqnarray*}
\Omega_{x,3,H}^{(\nu)}(y)&\geq &f(x)-\|\nabla f(x)\|_{*}\|y-x\|-\frac{1}{2}\|\nabla^{2}f(x)\|\|y-x\|^{2}-\frac{1}{6}\|D^{3}f(x)\|\|y-x\|^{3}\\
                &  &+\frac{H}{6}\|y-x\|^{3+\nu}.
\end{eqnarray*}
Thus, to ensure $\Omega_{x,3,H}^{(\nu)}(y)>A$, it is enough to have
\begin{equation*}
\frac{H}{6}\|y-x\|^{3+\nu}>(A-f(x))+\|\nabla f(x)\|_{*}\|y-x\|+\frac{1}{2}\|\nabla^{2}f(x)\|\|y-x\|^{2}+\frac{1}{6}\|D^{3}f(x)\|\|y-x\|^{3},
\end{equation*}
which is equivalent to 
\begin{equation}
\|y-x\|^{\nu}>\frac{6(A-f(x))}{H\|y-x\|^{3}}+\frac{6\|\nabla f(x)\|_{*}}{H\|y-x\|^{2}}+\frac{3\|\nabla^{2}f(x)\|}{H\|y-x\|}+\frac{\|D^{3}f(x)\|}{H}.
\label{eq:mais3.2}
\end{equation}
Note that, if (\ref{eq:mais3.1}) holds, then (\ref{eq:mais3.2}) holds. Therefore,
\begin{equation*}
\text{(\ref{eq:mais3.1})}\Longrightarrow \text{(\ref{eq:mais3.2})}\Longrightarrow\,\,\Omega_{x,3,H}^{(\nu)}(y)>A.
\end{equation*}
This proves statement (a). 

Finally, given $A>0$, if
\begin{equation*}
\|y\|>\|x\|+\max\left\{\left[\frac{6(A-f(x))}{H}\right]^{\frac{1}{3}},\left[\frac{6\|\nabla f(x)\|_{*}}{H}\right]^{\frac{1}{2}},\frac{3\|\nabla^{2}f(x)\|}{H},\left[3+\frac{\|D^{3}f(x)\|}{H}\right]^{\frac{1}{\nu}}\right\},
\end{equation*}
then, by (a), we have $\Omega_{x,3,H}^{(\nu)}(y)>A$. Since $A>0$ is arbitrary, we conclude that 
\begin{equation*}
\lim_{\|y\|\to +\infty}\,\Omega_{x,3,H}^{(\nu)}(y)=+\infty.
\end{equation*}  
This proves statement (b).
\end{proof}

\noindent As a corollary of Lemma \ref{lem:mais3.1}, we can establish the following upper bound for $\|y-x\|$ whenever $y$ belongs to the convex hull of a suitable sublevel set of $\Omega_{x,3,H}(\,.\,)$.

\begin{lemma}
\label{lem:mais3.2}
Given $x\in\E$, $H>0$ and $\nu\neq 0$, let 
\begin{equation}
\mathcal{L}_{H}(x)=\left\{z\in\E\,:\,\Omega_{x,3,H}^{(\nu)}(z)\leq f(x)\right\}.
\label{eq:mais3.3}
\end{equation}
Then,
\begin{equation}
\|y-x\|\leq\max\left\{1,\left[\frac{6\|\nabla f(x)\|_{*}}{H}\right]^{\frac{1}{2}},\frac{3\|\nabla^{2}f(x)\|}{H},\left[3+\frac{\|D^{3}f(x)\|}{H}\right]^{\frac{1}{\nu}}\right\}\equiv D_{x,H},
\label{eq:mais3.4}
\end{equation}
for all $y\in\text{co}\left(\mathcal{L}_{H}(x)\right)$. Consequently,
\begin{equation}
\sup\left\{\|\nabla^{2}\rho_{x}(y)\|\,:\,y\in\text{co}\left(\mathcal{L}_{H}(x)\right)\right\}\leq\|\nabla^{2}f(x)\|+(2+\nu)D_{x,H}^{2}\equiv \hat{N}_{x,H}.
\label{eq:mais3.5}
\end{equation}
\end{lemma}

\begin{proof}
By Lemma \ref{lem:mais3.1} (a) with $A=f(x)$, we have the implication
\begin{equation*}
\|y-x\|>\max\left\{\left[\frac{6\|\nabla f(x)\|_{*}}{H}\right]^{\frac{1}{2}},\frac{3\|\nabla^{2}f(x)\|}{H},\left[3+\frac{\|D^{3}f(x)\|}{H}\right]^{\frac{1}{\nu}}\right\}\Longrightarrow\,\Omega_{x,3,H}^{(\nu)}(y)>f(x),
\end{equation*}
whose contrapositive is
\begin{equation*}
\Omega_{x,3,H}^{(\nu)}(y)\leq f(x)\Longrightarrow \|y-x\|\leq \max\left\{\left[\frac{6\|\nabla f(x)\|_{*}}{H}\right]^{\frac{1}{2}},\frac{3\|\nabla^{2}f(x)\|}{H},\left[3+\frac{\|D^{3}f(x)\|}{H}\right]^{\frac{1}{\nu}}\right\}.
\end{equation*}
Thus, if $y\in\mathcal{L}_{H}(x)$, then the bound (\ref{eq:mais3.4}) holds for $y$. Consequently, as in the proof of Lemma 3.4, we obtain
\begin{equation}
\|y-x\|\leq D_{x,H},\quad\forall y\in\text{co}\left(\mathcal{L}_{H}(x)\right).
\label{eq:mais3.6}
\end{equation}
Finally, (\ref{eq:mais3.5}) follows by (\ref{eq:3.9}), $D_{x,H}\geq 1$ and (\ref{eq:mais3.6}).
\end{proof}

\subsection{Gradient Method and its Efficiency}

Let us consider the problem
\\[0.10cm]
\begin{mdframed}
\begin{equation}
\min_{y\in\E}\,\Omega_{x,3,H}^{(\nu)}(y)
\label{eq:5.1}
\end{equation}
\end{mdframed}
By Theorem 3.3, Remark \ref{rem:3.1} and Lemma \ref{lem:mais3.2}, it follows that:
\begin{itemize}
\item $\Omega_{x,3,H}^{(\nu)}(\,.\,)$ is $L_{H}$-smooth;
\item $\rho_{x}(\,.\,)$ is uniformly convex of degree $3+\nu$ with parameter $2^{-(1+\nu)}$;
\item $\rho_{x}(\,.\,)$ is twice-differentiable and $\|\nabla^{2}\rho_{x}(y)\|$ is bounded on $\text{co}\left(\mathcal{L}_{H}(x)\right)$.
\end{itemize}
This means that $\Omega_{x,3,H}^{(\nu)}(\,.\,)$ and $\rho_{x}(\,.\,)$ satisfy assumptions H1-H3 in Appendix A. Therefore, we can apply Algorithm A (see page 17) to solve (\ref{eq:5.1}). The Bregman distance corresponding to $\rho_{x}(\,.\,)$ is 
\begin{equation}
\beta_{\rho_{x}}(u,v)=\rho_{x}(v)-\rho_{x}(u)-\la\nabla\rho_{x}(u),v-u\ra.
\label{eq:5.2}
\end{equation}
Thus, Algorithm A applied to (\ref{eq:5.1}) can be rewritten as follows.
\begin{mdframed}
\noindent\textbf{Algorithm 1. Algorithm A applied to (\ref{eq:5.1})}\\
\noindent\textbf{Step 0.} Choose $L_{0}>0$. Set $y_{0}=x$ and $k:=0$.\\
\noindent\textbf{Step 1.} Set $i:=0$.\\ 
\noindent\textbf{Step 1.1.} Compute $y_{k,i}^{+}=\arg\min_{z\in\E}\left\{\la\nabla \Omega_{x,3,H}^{(\nu)}(y_{k}),z-y_{k}\ra+2^{i}L_{k}\beta_{\rho_{x}}(y_{k},z)\right\}$.\\
\noindent\textbf{Step 1.2.} If 
\begin{equation*}
\Omega_{x,3,H}^{(\nu)}(y_{k,i}^{+})\leq \Omega_{x,3,H}^{(\nu)}(y_{k})+\la\nabla\Omega_{x,3,H}^{(\nu)}(y_{k}),y_{k,i}^{+}-y_{k}\ra+2^{i}L_{k}\beta_{\rho_{x}}(y_{k},y_{k,i}^{+}),
\end{equation*}
set $i_{k}:=i$ and go to Step 2. Otherwise, set $i:=i+1$ and go to Step 1.1.\\
\noindent\textbf{Step 2.} Set $y_{k+1}=y_{k,i_{k}}^{+}$ and $L_{k+1}=2^{i_{k}-1}L_{k}$.\\
\noindent\textbf{Step 3.} Set $k:=k+1$ and go to Step 1. 
\end{mdframed}

When $H$ is sufficiently large, the next theorem establishes that Algorithm 1 takes at most $\mathcal{O}\left(\log(\epsilon^{-1})\right)$ iterations to find an $\epsilon$-stationary point of $\Omega_{x,3,H}^{(\nu)}(\,.\,)$.
\begin{theorem}
\label{thm:5.1}
Suppose that $f$ has a global minimizer $x^{*}$ and that $x\in\mathcal{F}(x_{0})$ with
\begin{equation}
\sup_{y\in\mathcal{F}(x_{0})}\|y-x^{*}\|\leq R_{0}<+\infty,\quad R_{0}\geq 1.
\label{eq:5.3}
\end{equation}
Denote $M_{H}=\max\left\{2L_{0},4L_{H}\right\}$, with $L_{H}$ defined in (\ref{eq:3.7}) and
\begin{equation}
N_{x}=\|\nabla^{2}f(x)\|+12R_{0}^{2}.
\label{eq:5.4}
\end{equation}
Let $\left\{y_{k}\right\}_{k\geq 0}$ be a sequence generated by Algorithm 1. If $H>[6/(3+\nu)]H_{f,3}(\nu)$ and $\|\nabla\Omega_{x,3,H}^{(\nu)}(y_{T+1})\|_{*}>\epsilon$ for a given $\epsilon\in (0,1)$, then
\begin{equation}
T\leq\left[\log_{2}\left(\dfrac{M_{H}}{M_{H}-\mu_{H}}\right)\right]^{-1}\left[C_{x,H}+(3+\nu)\right]\log_{2}(\epsilon^{-1}),
\label{eq:5.6}
\end{equation}
where
\begin{equation}
C_{x,H}=\log_{2}\left(\dfrac{4(3+\nu)M_{H}^{2+\nu}N_{x}^{3}\mu_{H}}{2^{-(1+\nu)}}\right).
\label{eq:5.7}
\end{equation}
\end{theorem}

\begin{proof}
Since $H>[6/(3+\nu)]H_{f,3}(\nu)$, it follows from Theorem 3.3 that $\Omega_{x,3,H}^{(\nu)}(\,.\,)$ is $L_{H}$-smooth and $\mu_{H}$-strongly convex relative to $\rho_{x}(\,.\,)$, with $\mu_{H}>1$. Moreover, by Remark \ref{rem:3.1} and Lemma \ref{lem:3.5}, function $\rho_{x}(\,.\,)$ is twice-differentiable, uniformly convex of degree $3+\nu$ with parameter $2^{-(1+\nu)}$, and satisfies
\begin{equation*}
\sup\left\{\|\nabla^{2}\rho_{x}(y)\|\,:\,y\in\text{co}\left(\mathcal{L}_{H}(x)\right)\right\}\leq N_{x}.
\end{equation*}
Thus, $\Omega_{x,3,H}^{(\nu)}(\,.\,)$ and $\rho_{x}(\,.\,)$ satisfy assumptions H1-H4 in Appendix A with $L=L_{H}$, $q=3+\nu$, $\sigma_{q}=2^{-(1+\nu)}$, $N=N_{x}$ and $\mu=\mu_{H}$. Consequently, by Corollary A.6, we must have
\begin{equation}
T\leq\left[\log_{2}\left(\dfrac{M_{H}}{M_{H}-\mu_{H}}\right)\right]^{-1}\left[\tilde{C}_{x,H}+(3+\nu)\right]\log_{2}(\epsilon^{-1}),
\label{eq:5.8}
\end{equation}
where 
\begin{equation}
\tilde{C}_{x,H}=\log_{2}\left(\dfrac{2(3+\nu)M_{H}^{2+\nu}N_{x}\mu_{H}\beta_{\rho_{x}}(x,S(x))}{2^{-(1+\nu)}}\right).
\label{eq:5.9}
\end{equation}
with $S(x)\in\arg\min_{y\in\E}\Omega_{x,3,H}^{(\nu)}(y)$. Clearly, $S(x)\in\mathcal{L}_{H}(x)$. Thus, if follows from  (\ref{eq:5.2}), (3.14), $R_{0}\geq 1$ and (\ref{eq:5.4}) that
\begin{eqnarray}
\beta_{\rho_{x}}(x,S(x))&=& \rho_{x}(S(x))\nonumber\\
                        &=&\frac{1}{2}\la\nabla^{2}f(x)(S(x)-x),S(x)-x\ra+\frac{1}{3+\nu}\|S(x)-x\|^{3+\nu}\nonumber\\
                        &\leq &\frac{1}{2}\|\nabla^{2}f(x)\|\|S(x)-x\|^{2}+\frac{1}{3+\nu}\|S(x)-x\|^{3+\nu}\nonumber\\
                        &\leq &\frac{1}{2}\left[\|\nabla^{2}f(x)\|+\|S(x)-x\|^{1+\nu}\right]\|S(x)-x\|^{2}\nonumber\\
                        &\leq &\frac{1}{2}\left[\|\nabla^{2}f(x)\|+(2R_{0})^{1+\nu}\right](2R_{0})^{2}\nonumber\\
                        & \leq & 2\left[\|\nabla^{2}f(x)\|+4R_{0}^{1+\nu}\right]R_{0}^{2}\nonumber\\
                        &\leq & 2N_{x}^{2}.
\label{eq:5.10}
\end{eqnarray}
Finally, combining (\ref{eq:5.8})-(\ref{eq:5.10}), we obtain (\ref{eq:5.6})-(\ref{eq:5.7}).
\end{proof}

\begin{remark}
If $x\notin\mathcal{F}(x_{0})$, by Lemma 3.6 we also have $T\leq\mathcal{O}\left(\log_{2}(\epsilon^{-1})\right)$ with $N_{x}$ replaced by $\hat{N}_{x,H}$ in (\ref{eq:5.7}), as long as $\nu\neq 0$. In both cases, it is worth mentioning that the potentially ``bad'' constants $N_{x}$ and $\hat{N}_{x,H}$ are inside the $\log_{2}(\,.\,)$ in (\ref{eq:5.7}).
\end{remark}

When $H\leq[6/(3+\nu)]H_{f,3}(\nu)$, problem (\ref{eq:5.1}) may be nonconvex. Even in this case, we can establish complexity bounds for Algorithm 1 if $\nu\neq 0$. 
\begin{theorem}
\label{thm:3.10}
Given $\epsilon\in (0,1)$, let $\left\{y_{k}\right\}_{k\geq 0}$ be a sequence generated by Algorithm 1 such that
\begin{equation}
\|\nabla\Omega_{x,3,H}^{(\nu)}(y_{k})\|_{*}>\epsilon\quad\text{for}\,\,k=0,\ldots,T.
\label{eq:5.11}
\end{equation}
Then, the following statements are true:\\[0.2cm]
\noindent (a) If $H<[6/(3+\nu)]H_{f,3}(\nu)$,  then
\begin{equation*}
T\leq\left[\dfrac{\hat{N}_{x,H}^{3+\nu}(3+\nu)M_{H}^{2+\nu}F_{x}}{2^{-(1+\nu)}}\right]\epsilon^{-(3+\nu)},
\end{equation*}
where $\hat{N}_{x,H}$ is defined in (\ref{eq:mais3.5}) and
\begin{equation*}
F_{x}=D_{x,H}\left[\|\nabla f(x)\|_{*}+\frac{1}{2}\|\nabla^{2}f(x)\|D_{x,H}+\|D^{3}f(x)\|D_{x,H}^{2}\right],
\end{equation*}
with $D_{x,H}$ given in (\ref{eq:mais3.4}).\\
\noindent (b) If $H=\left[6/(3+\nu)\right]H_{f,3}(\nu)$,  then
\begin{equation*}
T\leq 3\left[M_{H}\hat{N}_{x,H}\right]^{\frac{3+\nu}{2}}\left[\dfrac{(3+\nu)\hat{N}_{x,H}^{2}}{2^{-(2+\nu)}}\right]^{\frac{1}{2}}\epsilon^{-\frac{3+\nu}{2}}
\end{equation*}
\end{theorem}

\begin{proof}
Combining Theorem 3.3(a), Remark \ref{rem:3.1}, Lemma 3.6 and (\ref{eq:5.11}) with Theorem A.2, we obtain
\begin{equation}
T\leq\left[\dfrac{\hat{N}_{x,H}^{3+\nu}(3+\nu)M_{H}^{2+\nu}\left(f(x)-\Omega_{x,3,H}^{(\nu)}(S(x))\right)}{2^{-(1+\nu)}}\right]\epsilon^{-(3+\nu)},
\label{eq:5.12}
\end{equation}
with $S(x)\in\arg\min_{y\in\E}\Omega_{x,3,H}^{(\nu)}(y)$. Since $S(x)\in\mathcal{L}_{H}(x)$, it follows from (3.19) that
\begin{eqnarray}
f(x)-\Omega_{x,3,H}^{(\nu)}(S(x))&=&\la\nabla f(x),x-S(x)\ra+\frac{1}{2}\la\nabla^{2}f(x)(S(x)-x),S(x)-x\ra\nonumber\\
& &-\frac{1}{6}D^{3}f(x)[S(x)-x]^{3}-\frac{H}{6}\|S(x)-x\|^{3+\nu}\nonumber\\
&\leq &\|\nabla f(x)\|_{*}\|S(x)-x\|+\frac{1}{2}\|\nabla^{2}f(x)\|\|S(x)-x\|^{2}+\|D^{3}f(x)\|\|S(x)-x\|^{3}\nonumber\\
&\leq & D_{x,H}\left[\|\nabla f(x)\|_{*}+\frac{1}{2}\|\nabla^{2}f(x)\|D_{x,H}+\|D^{3}f(x)\|D_{x,H}^{2}\right]\nonumber\\
&=& F_{x}.
\label{eq:5.13}
\end{eqnarray}
Thus, from (\ref{eq:5.12}) and (\ref{eq:5.13}) we see that statement (a) is true.

Now, suppose that $H=[6/(3+\nu)]H_{f,3}(\nu)$. Then, by Theorem 3.3(b) functions $\Omega_{x,3,H}^{(\nu)}(\,.\,)$ and $\rho_{x}(\,.\,)$ satisfy assumption H4 in Appendix A with $\mu=0$. Consequently, by (\ref{eq:5.11}) and Corollary A.5 we have
\begin{equation}
T\leq 3\left[M_{H}\hat{N}_{x,H}\right]^{\frac{3+\nu}{2}}\left[\dfrac{(3+\nu)\beta_{\rho_{x}}(x,S(x))}{2^{-(1+\nu)}}\right]^{\frac{1}{2}}\epsilon^{-\frac{3+\nu}{2}}.
\label{eq:5.14}
\end{equation}
As in the proof of Theorem \ref{thm:5.1}, by (\ref{eq:mais3.5}) we have
\begin{equation}
\beta_{\rho_{x}}(x,S(x))\leq\frac{1}{2}\hat{N}_{x,H}^{2}.
\label{eq:5.15}
\end{equation}
Thus, combining (\ref{eq:5.14}) and (\ref{eq:5.15}), we see that statement (b) is also true.
\end{proof}

In view of Lemma 2.1 and Theorem \ref{thm:5.1}, if $H>[6/(3+\nu)]H_{f,3}(\nu)$, then Algorithm 1 takes at most $\mathcal{O}\left(\log_{2}(\epsilon^{-1})\right)$ iterations to generate $x^{+}$ such that either $\|\nabla f(x^{+})\|_{*}\leq\epsilon$ or (\ref{eq:2.4})-(\ref{eq:2.5}) holds for $p=3$. In contrast, by Theorem \ref{thm:3.10}, if $H=[6/(3+\nu)]H_{f,3}(\nu)$ or $H<[6/(3+\nu)]H_{f,3}(\nu)$, this iteration complexity bound is increased to $\mathcal{O}\left(\epsilon^{-\frac{3+\nu}{2}}\right)$ and $\mathcal{O}(\epsilon^{-(3+\nu)})$, respectively, in the case $\nu\neq 0$.

\section{Auxiliary Problems in Composite Minimization}

For third-order tensor methods designed to composite minimization \cite{GN4,JIA}, the auxiliary problems take the form:
\\[0.2cm]
\begin{mdframed}
\begin{equation}
\min_{y\in\E}\tilde{\Omega}_{x,3,H}^{(\nu)}(y)\equiv\Omega_{x,3,H}^{(\nu)}(y)+\varphi(y),
\label{eq:6.1}
\end{equation}
\end{mdframed}

\noindent where $\Omega_{x,3,H}^{(\nu)}(\,.\,)$ is defined by (\ref{eq:2.2}), $H\geq (p-1)H_{f,3}(\nu)$ and $\varphi:\E\to\mathbb{R}\cup\left\{+\infty\right\}$ is a simple closed convex function whose effective domain has nonempty relative interior. In this case, we are interested in finding an approximate solution $x^{+}$ for (\ref{eq:6.1}) such that\footnote{See, e.g., section 5 in \cite{GN4}.}
\begin{equation}
\tilde{\Omega}_{x,3,H}^{(\nu)}(x^{+})\leq f(x)+\varphi(x)\equiv \tilde{f}(x),
\label{eq:6.2}
\end{equation}
and
\begin{equation}
\|\nabla\Omega_{x,3,H}^{(\nu)}(x^{+})+g_{\varphi}(x^{+})\|_{*}\leq\theta\|x^{+}-x\|^{2+\nu},
\label{eq:6.3}
\end{equation}
for some $g_{\varphi}(x^{+})\in\partial\varphi(x^{+})$. For general $p\geq 2$, we have the following generalization of Lemma 2.1.
\begin{lemma}
\label{lem:6.1}
Let $x\in\E$, $H,\theta>0$ and $\delta\in (0,1)$. Given $g_{\varphi}(x^{+})\in\partial\varphi(x^{+})$, if 
\begin{equation}
\|\nabla f(x^{+})+g_{\varphi}(x^{+})\|_{*}\geq\delta,
\label{eq:6.4}
\end{equation}
and
\begin{equation}
\|\nabla\Omega_{x,p,H}^{(\nu)}(x^{+})+g_{\varphi}(x^{+})\|_{*}\leq\min\left\{\frac{1}{2},\dfrac{\theta (p-1)!}{2\left[H_{f,p}(\nu)+H(p+\nu)\right]}\right\}\delta,
\label{eq:6.5}
\end{equation}
then $x^{+}$ satisfies (\ref{eq:6.3}).
\end{lemma}

\begin{proof}
It follows as in the proof of Lemma 2.1.
\end{proof}

Suppose that $H\geq 2H_{f,3}(\nu)$. Then, in view of the relative smoothness properties of $\Omega_{x,3,H}^{(\nu)}(\,.\,)$ established in subsection 3.1, we can apply Algorithm B (see page 25) to solve (\ref{eq:6.1}):
\\
\begin{mdframed}
\noindent\textbf{Algorithm 2. Algorithm B applied to (\ref{eq:6.1})}\\
\noindent\textbf{Step 0.} Set $y_{0}=x$ and $k:=0$.\\
\noindent\textbf{Step 1.} Compute $y_{k+1}=\arg\min_{z\in\E}\left\{\la\nabla \Omega_{x,3,H}(y_{k}),z-y_{k}\ra+2L_{H}\beta_{\rho_{x}}(y_{k},z)\right\}$.\\
\noindent\textbf{Step 2.} Set $k:=k+1$ and go to Step 1. 
\end{mdframed}

The next theorem establishes that Algorithm 2 takes at most $\mathcal{O}\left(\log_{2}(\epsilon^{-1})\right)$ iterations to generate $x^{+}$ such that
\begin{equation*}
\|\nabla\Omega_{x,3,H}^{(\nu)}(x^{+})+g_{\varphi}(x^{+})\|_{*}\leq\epsilon,
\end{equation*}
with $g_{\varphi}(x^{+})\in\partial\varphi(x^{+})$.

\begin{theorem}
Suppose that $\tilde{f}=f+\varphi$ has a global minimizer $x^{*}$ and that
\begin{equation*}
x\in\tilde{\mathcal{F}}(x_{0})\equiv\left\{z\in\E\,:\,\tilde{f}(z)\leq\tilde{f}(x_{0})\right\},
\end{equation*}
with 
\begin{equation*}
\sup_{y\in\tilde{\mathcal{F}}(x_{0})}\|y-x^{*}\|\leq R_{0}<+\infty,\quad R_{0}\geq 1.
\end{equation*}
Assume that $H\geq 2H_{f,3}(\nu)$ and let $\left\{y_{k}\right\}_{k\geq 0}$ be a sequence generated by Algorithm 2. Then, for all $k\geq 1$, we have
\begin{equation}
g_{\varphi}(y_{k})\equiv 2L_{H}\left[\nabla\rho_{x}(y_{k-1})-\nabla\rho_{x}(y_{k})\right]-\nabla\Omega_{x,3,H}^{(\nu)}(y_{k-1})\in\partial\varphi(y_{k}).
\label{eq:6.6}
\end{equation}
Moreover, if 
\begin{equation}
\|\nabla\Omega_{x,3,H}^{(\nu)}(y_{T+1})+g_{\varphi}(y_{T+1})\|_{*}>\epsilon
\label{eq:6.7}
\end{equation}
for a given $\epsilon\in (0,1)$, then
\begin{equation}
T\leq\left[\log_{2}\left(\dfrac{2L_{H}}{2L_{H}-\mu_{H}}\right)\right]^{-1}\left[K_{x,H}+(3+\nu)\right]\log_{2}(\epsilon^{-1}),
\label{eq:6.8}
\end{equation}
where 
\begin{equation}
K_{x,H}=\log_{2}\left(\dfrac{4(3+\nu)(2L_{H})^{2+\nu}N_{x}^{3}\mu_{H}}{2^{-(1+\nu)}}\right)
\end{equation}
with $N_{x}$ given in (\ref{eq:5.4}).
\end{theorem}

\begin{proof}
By Lemma A.8 and $\text{ri}\left(\dom\varphi\right)\neq\emptyset$, we have
\begin{equation*}
u(y_{k})\equiv g_{\varphi}(y_{k})+\nabla\Omega_{x,3,H}^{(\nu)}(y_{k})\in\partial\tilde{\Omega}_{x,3,H}^{(\nu)}(y_{k})=\left\{\nabla\Omega_{x,3,H}^{(\nu)}(y_{k})\right\}+\partial\varphi(y_{k}).
\end{equation*}
Thus, $g_{\varphi}(y_{k})=u(y_{k})-\nabla\Omega_{x,3,H}^{(\nu)}(y_{k})\in\partial\varphi(y_{k})$, and so (\ref{eq:6.6}) holds. Moreover, by (\ref{eq:6.7}), we have 
\begin{equation*}
\|u(y_{T+1})\|_{*}>\epsilon.
\end{equation*}
Then, the bound (\ref{eq:6.8}) on $T$ follows directly from Corollary A.10.
\end{proof}

In view of Theorem 4.2, if $H\geq 2H_{f,3}(\nu)$, Algorithm 2 takes at most $\mathcal{O}\left(\log_{2}(\epsilon^{-1})\right)$ iterations to generate $x^{+}$ such that either $\|\nabla f(x^{+})+g_{\varphi}(x^{+})\|_{*}\leq\epsilon$ or (\ref{eq:6.4})-(\ref{eq:6.5}) holds, for $g_{\varphi}(x^{+})\in\partial\varphi(x^{+})$ defined in (\ref{eq:6.6}).

\section{Conclusion}

In this paper we studied the auxiliary problems that appear in non-universal adaptive $p$-order tensor methods for unconstrained minimization of convex functions with H\"{o}lder continuous $p$th derivatives \cite{GN3,GN4}. For $p=3$, we consider the use of Gradient Methods with Bregman Distance. When the regularization parameter is sufficiently large, we prove that Bregman Gradient Methods applied to the corresponding tensor model takes at most $\mathcal{O}(\log(\epsilon^{-1}))$ iterations to find either a suitable approximate stationary point of the tensor model or an $\epsilon$-approximate stationary point of the original objective function. The authors believe this work is a step towards implementable third-order tensor methods for convex optimization. Future research includes the development of methods for the auxiliary problems in universal tensor methods and numerical experimentation.

\subsection*{Acknowledgements}

The authors are very grateful to the two anonymous referees, whose comments helped to improve the first version of this paper.

\subsection*{Funding}

This project has received funding from the European Research Council (ERC) under the European Union’s Horizon 2020 research and innovation programme (grant agreement No. 788368).

\appendix

\section{Adaptive Bregman Proximal Gradient Method}

\subsection{Smooth Minimization}

Consider the following optimization problem
\begin{equation}
\min_{y\in\E}\,g(y),
\label{eq:4.1}
\end{equation}
where $g:\E\to\mathbb{R}$ is $L$-smooth relative to a convex and smooth function $d(\,.\,)$, that is, for all $x,y\in\E$,
\begin{equation}
g(x)\leq g(y)+\la \nabla g(y),x-y\ra +L\beta_{d}(y,x),
\label{eq:4.2}
\end{equation}
with
\begin{equation}
\beta_{d}(y,x):=d(x)-d(y)-\la \nabla d(y),x-y\ra 
\label{eq:4.3}
\end{equation}
being the Bregman distance corresponding to $d(\,.\,)$. We assume that $g(\,.\,)$ has at least one global minimizer $y^{*}\in\E$. We do not assume the convexity of $g(\,.\,)$ yet. 

We shall consider the following adaptive version of the Proximal Gradient Scheme proposed in \cite{LU}:
\begin{mdframed}
\noindent\textbf{Algorithm A. Adaptive Proximal Gradient Method}
\\[0.2cm]
\noindent\textbf{Step 0.} Choose $y_{0}\in\mathbb{E}$, $L_{0}>0$ and set $k:=0$.\\
\noindent\textbf{Step 1.} Set $i:=0$.\\ 
\noindent\textbf{Step 1.1.} Compute 
\begin{equation}
y_{k,i}^{+}=\arg\min_{x\in\E}\left\{\la\nabla g(y_{k}),x-y_{k}\ra+2^{i}L_{k}\beta_{d}(y_{k},x)\right\}.
\label{eq:4.4}
\end{equation}
\noindent\textbf{Step 1.2.} If 
\begin{equation}
g(y_{k,i}^{+})\leq g(y_{k})+\la\nabla g(y_{k}),y_{k,i}^{+}-y_{k}\ra+2^{i}L_{k}\beta_{d}(y_{k},y_{k,i}^{+}),
\label{eq:4.5}
\end{equation}
set $i_{k}:=i$ and go to Step 2. Otherwise, set $i:=i+1$ and go to Step 1.1.\\
\noindent\textbf{Step 2.} Set $y_{k+1}=y_{k,i_{k}}^{+}$ and $L_{k+1}=2^{i_{k}-1}L_{k}$.\\
\noindent\textbf{Step 3.} Set $k:=k+1$ and go to Step 1. 
\end{mdframed}

Let us assume that:
\begin{itemize}
\item[\textbf{H1.}] $g(\,.\,)$ is $L$-smooth relative to $d(\,.\,)$.
\item[\textbf{H2.}] $d(\,.\,)$ is twice differentiable and uniformly convex of degree $q$, with parameter $\sigma_{q}>0$.
\item[\textbf{H3.}] There exists a constant $N>0$ such that 
\begin{equation*}
\sup\left\{\|\nabla^{2}d(y)\|\,:\,y\in\text{co}\left(\mathcal{L}(y_{0})\right)\right\}\leq N,
\end{equation*}
where $\mathcal{L}(y_{0})=\left\{y\in\E\,:\,g(y)\leq g(y_{0})\right\}$.
\end{itemize}
The next lemma gives a global upper bound on $L_{k}$ and a lower bound on the functional decrease in successive iterations.

\begin{lemma}
\label{lem:4.1}
Suppose that H1 holds and let $\left\{y_{k}\right\}_{k\geq 0}$ be a sequence generated by Algorithm A. Then, for all $k$,
\begin{equation}
L_{k}\leq\max\left\{L_{0},2L\right\},
\label{eq:4.6}
\end{equation}
and
\begin{equation}
g(y_{k})-g(y_{k+1})\geq 2L_{k+1}\beta_{d}(y_{k+1},y_{k}).
\label{eq:4.7}
\end{equation}
\end{lemma}
\begin{proof}
Let us prove by induction that (\ref{eq:4.6}) is true. It is obvious for $k=0$. Assume that (\ref{eq:4.6}) is true for some $k\geq 0$. Then, it follows from H1 and (\ref{eq:4.2}) that $2^{i_{k}}L_{k}$ cannot be bigger than $4L$, since otherwise the line search procedure should have stopped earlier. Thus,
\begin{equation*}
L_{k+1}=\max\left\{L_{0},2^{i_{k}-1}L_{k}\right\}\leq \max\left\{L_{0},2L\right\},
\end{equation*}
that is, (\ref{eq:4.6}) also holds for $k+1$, which concludes the induction argument.

Now, let us prove (\ref{eq:4.7}). In view of (\ref{eq:4.4}), we have
\begin{equation*}
\nabla g(y_{k})+2^{i_{k}}L_{k}\left(\nabla d(y_{k+1})-\nabla d(y_{k})\right)=0,
\end{equation*}
which gives 
\begin{equation}
\la\nabla g(y_{k}),y_{k+1}-y_{k}\ra=-2^{i_{k}}L_{k}\la\nabla d(y_{k+1})-\nabla d(y_{k}),y_{k+1}-y_{k}\ra.
\label{eq:4.8}
\end{equation}
Then, combining (\ref{eq:4.5}) and (\ref{eq:4.8}), we get
\begin{eqnarray*}
g(y_{k+1})&\leq & g(y_{k})-2^{i_{k}}L_{k}\la\nabla d(y_{k+1})-\nabla d(y_{k}),y_{k+1}-y_{k}\ra +2^{i_{k}}L_{k}\beta_{d}(y_{k},y_{k+1})\\
& = & g(y_{k})-2^{i_{k}}L_{k}\la\nabla d(y_{k+1})-\nabla d(y_{k}),y_{k+1}-y_{k}\ra\\
& & + 2^{i_{k}}L_{k}\left[d(y_{k+1})-d(y_{k})-\la\nabla d(y_{k}),y_{k+1}-y_{k}\ra\right]\\
&=& g(y_{k})-2^{i_{k}}L_{k}\left[d(y_{k})-d(y_{k+1})-\la\nabla d(y_{k+1}),y_{k}-y_{k+1}\ra\right]\\
&=& g(y_{k})-2^{i_{k}}L_{k}\beta_{d}(y_{k+1},y_{k}),
\end{eqnarray*}
that is
\begin{equation}
g(y_{k})-g(y_{k+1})\geq 2^{i_{k}}L_{k}\beta_{d}(y_{k+1},y_{k}).
\label{eq:4.9}
\end{equation}
Finally, since $L_{k+1}=2^{i_{k}-1}L_{k}$, (\ref{eq:4.7}) follows directly from (\ref{eq:4.9}).
\end{proof}

\begin{theorem}
\label{thm:4.2}
Suppose that H1-H3 hold. Then, for all $k\geq 0$ we have
\begin{equation}
g(y_{k})-g(y_{k+1})\geq \dfrac{\sigma_{q}}{q[\max\left\{2L_{0},4L\right\}]^{q-1}N^{q}}\|\nabla g(y_{k})\|_{*}^{q},
\label{eq:4.10}
\end{equation}
where $\sigma_{q}$ and $N$ are specified in H2 and H3, respectively. Moreover, for all $T\geq 1$,
\begin{equation}
\min_{0\leq k\leq T-1}\|\nabla g(y_{k})\|_{*}\leq N\left[\dfrac{q\left[\max\left\{2L_{0},4L\right\}\right]^{q-1}(g(y_{0})-g(y^{*}))}{\sigma_{q}}\right]^{\frac{1}{q}}\left(\dfrac{1}{T}\right)^{\frac{1}{q}}.
\label{eq:4.12}
\end{equation}
Consequently, if
\begin{equation}
\nabla g(y_{k})\|_{*}>\epsilon,\quad\text{for}\,\,k=0,\ldots,T-1,
\label{eq:mais4.13}
\end{equation}
for a given $\epsilon>0$, we have
\begin{equation}
T\leq\left[\dfrac{N^{q}q\left[\max\left\{2L_{0},4L\right\}\right]^{q-1}(g(y_{0})-g(y^{*}))}{\sigma_{q}}\right]\epsilon^{-q}.
\label{eq:mais4.14}
\end{equation}
\end{theorem}
\begin{proof}
By H2, $d(\,.\,)$ is uniformly convex of degree $q$ with parameter $\sigma_{q}>0$. Therefore,
\begin{equation*}
\beta_{d}(y_{k+1},y_{k})\geq\dfrac{\sigma_{q}}{q}\|y_{k+1}-y_{k}\|^{q}.
\end{equation*}
In this case, by (\ref{eq:4.7}) we obtain
\begin{equation}
g(y_{k})-g(y_{k+1})\geq\dfrac{2L_{k+1}\sigma_{q}}{q}\|y_{k+1}-y_{k}\|^{q}.
\label{eq:4.13}
\end{equation}
By the definition of $y_{k+1}$, this point satisfies the following first-order optimality condition:
\begin{equation}
\nabla g(y_{k})+2^{i_{k}}L_{k}\left(\nabla d(y_{k+1})-\nabla d(y_{k})\right)=0.
\label{eq:4.14}
\end{equation}
In view of H3, it follows from the mean value theorem that $\nabla d$ is $N$-Lipschitz continous on $\text{co}\left(\mathcal{L}(y_{0})\right)$. From (\ref{eq:4.13}), we see that $\left\{g(y_{k})\right\}_{k\geq 0}$ is nonincreasing, and so $\left\{y_{k}\right\}\subset\mathcal{L}(y_{0})$. Combining these facts, we get 
\begin{equation}
\|\nabla d(y_{k+1})-\nabla d(y_{k})\|_{*}\leq N\|y_{k+1}-y_{k}\|,\,\,\forall k.
\label{eq:4.15}
\end{equation}
Then, it follows from (\ref{eq:4.14}), (\ref{eq:4.15}) and (\ref{eq:4.6}) that
\begin{eqnarray*}
\|\nabla g(y_{k})\|_{*}&\leq &2^{i_{k}}L_{k}\|\nabla d(y_{k+1})-\nabla d(y_{k})\|_{*}\leq (2^{i_{k}}L_{k})N\|y_{k+1}-y_{k}\|.
\end{eqnarray*}
Thus, 
\begin{equation}
\|y_{k+1}-y_{k}\|\geq\dfrac{1}{2L_{k+1}N}\|\nabla g(y_{k})\|_{*}.
\label{eq:4.16}
\end{equation}
Combining (\ref{eq:4.13}), (\ref{eq:4.16}) and (\ref{eq:4.6}), we obtain
\begin{eqnarray*}
g(y_{k})-g(y_{k+1})&\geq &\dfrac{2L_{k+1}\sigma_{q}}{q}\|y_{k+1}-y_{k}\|^{q}\\
&\geq & \dfrac{2L_{k+1}\sigma_{q}}{q}\dfrac{1}{(2L_{k+1})^{q}N^{q}}\|\nabla g(y_{k})\|_{*}^{q}\\
& = & \dfrac{\sigma_{q}}{q(2L_{k+1})^{q-1}N^{q}}\|\nabla g(y_{k})\|_{*}^{q}\\
&\geq &\dfrac{\sigma_{q}}{q\left[2\max\left\{L_{0},2L\right\}\right]^{q-1}N^{q}}\|\nabla g(y_{k})\|_{*}^{q},
\end{eqnarray*}
which gives (\ref{eq:4.10}). Summing up inequalities (\ref{eq:4.10}) for $k=0,\ldots,T-1$, we get
\begin{eqnarray*}
g(y_{0})-g(y^{*})&\geq & g(y_{0})-g(y_{T})\\
                 & =   &\sum_{k=0}^{T-1}g(y_{k})-g(y_{k+1})\\
                 &\geq &\sum_{k=0}^{T-1}\dfrac{\sigma_{q}}{q\left[2\max\left\{L_{0},2L\right\}\right]^{q-1}N^{q}}\|\nabla g(y_{k})\|^{q}\\
                 &\geq & T\dfrac{\sigma_{q}}{q\left[2\max\left\{L_{0},2L\right\}\right]^{q-1}N^{q}}\left(\min_{0\leq k\leq T-1}\|\nabla g(y_{k})\|_{*}\right)^{q},
\end{eqnarray*}
which gives (\ref{eq:4.12}). Finally, (\ref{eq:mais4.14}) follows directly from (\ref{eq:4.12}) and (\ref{eq:mais4.13}).
\end{proof}

Now, let us consider the following additional assumption:
\begin{itemize}
\item[\textbf{H4.}] $g(\,.\,)$ is $\mu$-strongly convex relative to $d(\,.\,)$.
\end{itemize}

\begin{lemma}(Three-Point Property)
Let $\phi(\,.\,)$ and $d(\,.\,)$ be convex functions and let $\beta_{d}(\,.\,,\,.\,)$ be the Bregman distance from $d(\,.\,)$. Given $y\in\E$, let
\begin{equation*}
y^{+}=\arg\min_{x\in\E}\left\{\phi(x)+\beta_{d}(y,x)\right\}.
\end{equation*}
Then, 
\begin{equation}
\phi(x)+\beta_{d}(y,x)\geq \phi(y^{+})+\beta_{d}(y,y^{+})+\beta_{d}(y^{+},x),\quad\forall x\in\E.
\label{eq:4.17}
\end{equation}
\label{lem:4.3}
\end{lemma}

\begin{proof}
See \cite{TS,CHEN,LAN}.
\end{proof}

The next theorem establishes sublinear and linear convergence rates for Algorithm A. 
\begin{theorem}
Suppose that H1, H2 and H4 hold and let $\left\{y_{k}\right\}_{k\geq 0}$ be a sequence generated by Algorithm A. Then,
\begin{equation}
g(y_{k})-g(y^{*})\leq \dfrac{\mu\beta_{d}(y_{0},y^{*})}{\left(1+\dfrac{\mu}{\max\left\{2L_{0},4L\right\}-\mu}\right)^{k}-1}\leq\dfrac{\left(\max\left\{2L_{0},4L\right\}-\mu\right)\beta_{d}(y_{0},y^{*})}{k},
\label{eq:4.18}
\end{equation}
where, in the case $\mu=0$, the middle expression is defined in the limit as $\mu\to 0^{+}$.
\label{thm:4.4}
\end{theorem}

\begin{proof}
By H1 and Lemma \ref{lem:4.1}, it follows that $\left\{y_{k}\right\}_{k\geq 0}$ is well-defined. Let us denote $M_{k}=2^{i_{k}}L_{k}$. Then, for all $k\geq 1$, it follows from (\ref{eq:4.5}) that 
\begin{equation}
g(y_{k})\leq g(y_{k-1})+\la\nabla g(y_{k-1}),y_{k}-y_{k-1}\ra +M_{k}\beta_{d}(y_{k-1},y_{k}).
\label{eq:4.19}
\end{equation}
In order to get an upper bound for the inner product in (\ref{eq:4.19}), let us apply Lemma \ref{lem:4.3} with $h=d$ and 
\begin{equation*}
\phi(x)=\dfrac{1}{M_{k}}\la\nabla g(y_{k-1}),x-y_{k-1}\ra.
\end{equation*}
In this case, $y^{+}=y_{k}$ and, for $y=y_{k-1}$, we obtain
\begin{equation*}
\phi(x)+\beta_{d}(y_{k-1},x)\geq\phi(y_{k})+\beta_{d}(y_{k-1},y_{k})+\beta_{d}(y_{k},x),\quad\forall x\in\E,
\end{equation*}
that is
\begin{equation*}
\la\nabla g(y_{k-1}),x-y_{k-1}\ra + M_{k}\beta_{d}(y_{k-1},x)\geq\la\nabla g(y_{k-1}),y_{k}-y_{k-1}\ra +M_{k}\beta_{d}(y_{k-1},y_{k})+M_{k}\beta_{d}(y_{k},x).
\end{equation*}
This gives the upper bound
\begin{eqnarray}
\la\nabla g(y_{k-1}),y_{k}-y_{k-1}\ra &\leq &\la\nabla g(y_{k-1}),x-y_{k-1}\ra +M_{k}\beta_{d}(y_{k-1},x)\nonumber\\
                                      &     & -M_{k}\beta_{d}(y_{k-1},y_{k})-M_{k}\beta_{d}(y_{k},x).
\label{eq:4.20}
\end{eqnarray}
Combining (\ref{eq:4.19}) and (\ref{eq:4.20}), we obtain
\begin{equation}
g(y_{k})\leq g(y_{k-1})+\la\nabla g(y_{k-1}),x-y_{k-1}\ra +M_{k}\beta_{d}(y_{k-1},x)-M_{k}\beta_{d}(y_{k},x).
\label{eq:4.21}
\end{equation}
By A4, we have
\begin{equation*}
g(x)\geq g(y_{k-1})+\la\nabla g(y_{k-1}),x-y_{k-1}\ra +\mu\beta_{d}(y_{k-1},x),
\end{equation*}
and so
\begin{equation}
\la\nabla g(y_{k-1}),x-y_{k-1}\ra\leq g(x)-g(y_{k-1})-\mu\beta_{d}(y_{k-1},x).
\label{eq:4.22}
\end{equation}
Now, using inequality (\ref{eq:4.22}) in (\ref{eq:4.21}), it follows that
\begin{equation*}
g(y_{k})\leq g(x)+(M_{k}-\mu)\beta_{d}(y_{k-1},x)-M_{k}\beta_{d}(y_{k},x).
\end{equation*}
Substituting $x=y^{*}$, we get
\begin{equation}
g(y_{k})\leq g(y^{*})+(M_{k}-\mu)\beta_{d}(y_{k-1},y^{*})-M_{k}\beta_{d}(y_{k},y^{*}).
\label{eq:4.23}
\end{equation}
Since $\beta_{d}(y_{k-1},y^{*})\geq 0$ and $\mu\geq 0$, it follows that
\begin{eqnarray*}
0&\leq & g(y_{k})-g(y^{*})\leq (M_{k}-\mu)\beta_{d}(y_{k-1},y^{*})-M_{k}\beta_{d}(y_{k},y^{*})\\
 &\leq & M_{k}\left[\beta_{d}(y_{k-1},y^{*})-\beta_{d}(y_{k},y^{*})\right]
\end{eqnarray*}
and so
\begin{equation}
\beta_{d}(y_{k-1},y^{*})-\beta_{d}(y_{k},y^{*})\geq 0.
\label{eq:4.24}
\end{equation}
Moreover, by Lemma \ref{lem:4.1} we have
\begin{equation}
M_{k}=2^{i_{k}}L_{k}=2(2^{i_{k}-1}L_{k})\leq 2L_{k+1}\leq\max\left\{2L_{0},4L\right\}.
\label{eq:4.25}
\end{equation}
Denote $M=\max\left\{2L_{0},4L\right\}$. In view of (\ref{eq:4.23})-(\ref{eq:4.25}), we obtain
\begin{eqnarray}
g(y_{k})&\leq & g(y^{*})+(M_{k}-\mu)\beta_{d}(y_{k-1},y^{*})-M_{k}\beta_{d}(y_{k},y^{*})\nonumber\\
        & =   & g(y^{*})+M_{k}\left[\beta_{d}(y_{k-1},y^{*})-\beta_{d}(y_{k},y^{*})\right]-\mu\beta_{d}(y_{k-1},y^{*})\nonumber\\
        &\leq & g(y^{*})+M\left[\beta_{d}(y_{k-1},y^{*})-\beta_{d}(y_{k},y^{*})\right]-\mu\beta_{d}(y_{k-1},y^{*})\nonumber\\
        &=& g(y^{*})+(\tilde{M}-\mu)\beta_{d}(y_{k-1},y^{*})-M\beta_{d}(y_{k},y^{*}).
\label{eq:4.26}
\end{eqnarray}
Now, as in the proof of Theorem 3.1 in \cite{LU}, we can show by induction that, for all $k\geq 1$,
\begin{eqnarray}
\sum_{i=1}^{k}\left(\dfrac{M}{M-\mu}\right)^{i}g(y_{i})&\leq &\sum_{i=1}^{k}\left(\dfrac{M}{M-\mu}\right)^{i}g(y^{*})+M\beta_{d}(y_{0},y^{*})\nonumber\\
& &-\left(\dfrac{M}{M-\mu}\right)^{k}M\beta_{d}(y_{k},y^{*}).
\label{eq:4.27}
\end{eqnarray}
Since $\left\{g(y_{k})\right\}$ is nonincreasing and $\beta_{d}(y_{k},y^{*})$ is nonnegative, it follows from (\ref{eq:4.27}) that
\begin{equation*}
\left[\sum_{i=1}^{k}\left(\dfrac{\tilde{M}}{M-\mu}\right)^{i}\right](g(y_{k})-g(y^{*}))\leq M\beta_{d}(y_{0},y^{*}),\quad\forall k\geq 1.
\end{equation*}
Thus, denoting
\begin{equation*}
C_{k}=\dfrac{1}{\sum_{i=1}^{k}\left(\dfrac{M}{M-\mu}\right)^{i}}
\end{equation*}
we get
\begin{equation}
g(y_{k})-g(y^{*})\leq C_{k}M\beta_{d}(y_{0},y^{*}),\quad\forall k\geq 1.
\label{eq:4.28}
\end{equation}
If $\mu=0$, it follows that $C_{k}=1/k$ and so (\ref{eq:4.28}) becomes
\begin{equation}
g(y_{k})-g(y^{*})\leq\dfrac{M}{k}\beta_{d}(y_{0},y^{*}).
\label{eq:4.29}
\end{equation}
On the other hand, if $\mu>0$ we have
\begin{equation*}
\sum_{i=1}^{k}\left(\dfrac{M}{M-\mu}\right)^{i}=\dfrac{\left(\dfrac{M}{M-\mu}\right)\left[\left(\dfrac{M}{M-\mu}\right)^{k}-1\right]}{\left(\dfrac{M}{M-\mu}\right)-1}=\dfrac{M\left[\left(1+\dfrac{\mu}{M-\mu}\right)^{k}-1\right]}{\mu}
\end{equation*}
which gives
\begin{equation}
C_{k}=\dfrac{\mu}{M\left[\left(1+\dfrac{\mu}{M-\mu}\right)^{k}-1\right]}.
\label{eq:4.30}
\end{equation}
In this case, combining (\ref{eq:4.28}) and (\ref{eq:4.30}) we obtain
\begin{equation}
g(y_{k})-g(y^{*})\leq\dfrac{\mu\beta_{d}(y_{0},y^{*})}{\left[\left(1+\dfrac{\mu}{M-\mu}\right)^{k}-1\right]}.
\label{eq:4.31}
\end{equation}
Thus, (\ref{eq:4.18}) follows from (\ref{eq:4.29}), (\ref{eq:4.31}) and $M=\max\left\{2L_{0},4L\right\}$.
\end{proof}

\begin{corollary}
\label{cor:4.5}
Suppose that H1-H3 hold and let $\left\{y_{k}\right\}_{k\geq 0}$ be a sequence generated by Algorithm A. Additionaly, assume that H4 holds with $\mu=0$. If $T=3s$ for some $s\geq 1$, then
\begin{equation}
\min_{0\leq k\leq T-1}\|\nabla g(y_{k})\|_{*}\leq MN\left[\dfrac{q\beta_{d}(y_{0},y^{*})}{\sigma_{q}}\right]^{\frac{1}{q}}\left(\dfrac{3}{T}\right)^{\frac{2}{q}},
\label{eq:4.32}
\end{equation}
where $M=\max\left\{2L_{0},4L\right\}$. Consequently, if
\begin{equation}
\|\nabla g(y_{k})\|_{*}>\epsilon\quad\text{for}\,\,k=0,\ldots,T-1,
\label{eq:mais4.33}
\end{equation}
for a given $\epsilon>0$, then
\begin{equation}
T\leq 3\left[\max\left\{2L_{0},4L\right\}N\right]^{\frac{q}{2}}\left[\dfrac{q\beta_{d}(y_{0},y^{*})}{\sigma_{q}}\right]^{\frac{1}{2}}\epsilon^{-\frac{q}{2}}.
\label{eq:mais4.34}
\end{equation}
\end{corollary}

\begin{proof}
By Theorem \ref{thm:4.4}, we have
\begin{equation*}
g(y_{i})-g(y^{*})\leq\dfrac{M\beta_{d}(y_{0},y^{*})}{i},\,\,\forall i\geq 1.
\end{equation*}
Since $T=3s$, in particular, it follows that
\begin{eqnarray*}
\dfrac{M\beta_{d}(y_{0},y^{*})}{2s}&\geq & g(y_{2s})-g(y^{*})\\
           & = & g(y_{T})-g(y^{*})+\sum_{k=2s}^{T-1}\left[g(y_{k})-g(y_{k+1})\right]\\
           &\geq & s\dfrac{\sigma_{q}}{qM^{q-1}N^{q}}\left(\min_{0\leq k\leq T-1}\|\nabla g(y_{k})\|_{*}\right)^{q}.
\end{eqnarray*}
Therefore,
\begin{equation*}
\left(\min_{0\leq k\leq T-1}\|\nabla g(y_{k})\|_{*}\right)^{q}\leq\left[\dfrac{q(MN)^{q}\beta_{d}(y_{0},y^{*})}{\sigma_{q}}\right]\dfrac{1}{s^{2}}
\end{equation*}
which gives (\ref{eq:4.32}). Finally, (\ref{eq:mais4.34}) follows directly from (\ref{eq:4.23}) and (\ref{eq:mais4.33}).
\end{proof}

\begin{corollary}
Suppose that H1-H3 hold and let $\left\{y_{k}\right\}_{k\geq 0}$ be a sequence generated by Algorithm A. Additionaly, assume that H4 holds with $\mu>0$. Then, for all $T\geq\left[\log_{2}\left(1+\dfrac{\mu}{M-\mu}\right)\right]^{-1}$, with $M=\max\left\{2L_{0},4L\right\}$, we have
\begin{equation}
\|\nabla g(y_{T})\|_{*}\leq\left[\dfrac{2qM^{q-1}N\mu\beta_{d}(y_{0},y^{*})}{\sigma_{q}}\right]^{\frac{1}{q}}\left(\dfrac{1}{1+\frac{\mu}{M-\mu}}\right)^{\frac{T}{q}}
\label{eq:4.33}
\end{equation}
Consequently, if $\|\nabla g(y_{T})\|_{*}>\epsilon$, for a given $\epsilon\in (0,1)$, then
\begin{equation}
T\leq \left[\log_{2}\left(\dfrac{\max\left\{2L_{0},4L\right\}}{\max\left\{2L_{0},4L\right\}-\mu}\right)\right]^{-1}\left[C+q\right]\log_{2}(\epsilon^{-1}),
\label{eq:extra4.34}
\end{equation}
where
\begin{equation}
C=\log_{2}\left(\dfrac{2q\max\left\{2L_{0},4L\right\}^{q-1}N\mu\beta_{d}(y_{0},y^{*})}{\sigma_{q}}\right).
\label{eq:extra4.35}
\end{equation}
\label{cor:4.6}
\end{corollary}

\begin{proof}
By Theorems \ref{thm:4.2} and \ref{thm:4.4}, for all $k\geq 1$ we have
\begin{eqnarray*}
\dfrac{\sigma_{q}}{qM^{q-1}N^{q}}\|\nabla g(y_{k})\|_{*}^{q}&\leq & g(y_{k})-g(y^{*})\\
&\leq &\dfrac{\mu\beta_{d}(y_{0},y^{*})}{\left[\left(1+\frac{\mu}{M-\mu}\right)^{k}-1\right]}
\end{eqnarray*}
In particular, it follows that
\begin{equation}
\|\nabla g(y_{T})\|_{*}^{q}\leq\dfrac{qM^{q-1}N^{q}\mu\beta_{d}(y_{0},y^{*})}{\sigma_{q}\left[\left(1+\frac{\mu}{M-\mu}\right)^{T}-1\right]}
\label{eq:4.34}
\end{equation}
Since $T\geq\left[\log_{2}\left(1+\frac{\mu}{M-\mu}\right)\right]^{-1}$ we have
\begin{equation}
\left(1+\frac{\mu}{M-\mu}\right)^{T}-1\geq\dfrac{1}{2}\left(1+\frac{\mu}{M-\mu}\right)^{T}.
\label{eq:4.35}
\end{equation}
Thus, combining (\ref{eq:4.34}) and (\ref{eq:4.35}), it follows that
\begin{equation*}
\|\nabla g(y_{T})\|_{*}^{q}\leq\dfrac{2qM^{q-1}N^{q}\mu\beta_{d}(y_{0},y^{*})}{\sigma_{q}\left(1+\frac{\mu}{M-\mu}\right)^{T}},
\end{equation*}
which gives (\ref{eq:4.33}). Finally, (\ref{eq:extra4.34}) follows directly from (\ref{eq:4.33}), $\|g(y_{T})\|_{*}>\epsilon$ and $\epsilon\in (0,1)$.
\end{proof}

In summary, if $g(\,.\,)$ is $L$-smooth relative to a convex function $d(\,.\,)$ which is uniformly convex of degree $q$, then Algorithm A takes at most $\mathcal{O}(\delta^{-q})$ iterations to generate a point $y_{k}$ such that $\|\nabla g(y_{k})\|\leq\delta$. If $g(\,.\,)$ is also $\mu$-strongly convex relative to $d(\,.\,)$ with $\mu=0$, then this complexity bound is reduced to $\mathcal{O}(\delta^{-q/2})$. Moreover, if $\mu>0$, the complexity bound is further improved to $\mathcal{O}(\log(\delta^{-1}))$. 

\subsection{Composite Minimization}

Consider now the composite minimization problem
\begin{equation}
\min_{y\in\E}\,\tilde{g}(y)\equiv g(y)+\varphi(y),
\label{eq:A41}
\end{equation}
where $g:\E\to\mathbb{R}$ is a twice-differentiable function satisfying H1 and H4 (on pages 17 and 20, respectively), and $\varphi:\E\to\mathbb{R}\cup\left\{+\infty\right\}$ is a simple closed convex function whose effective domain has nonempty relative interior. We assume that there exists at least one optimal solution $y^{*}\in\E$ for (\ref{eq:A41}). Moreover, for the sake of brevity, we suppose that the constant $L$ in H1 is known. Thus, to approximately solve (\ref{eq:A41}), we may use the following adaptation of Algorithm A:
\\[0.2cm]
\begin{mdframed}
\noindent\textbf{Algorithm B. Proximal Gradient Method}
\\[0.2cm]
\noindent\textbf{Step 0.} Choose $y_{0}\in\mathbb{E}$ and set $k:=0$.\\
\noindent\textbf{Step 1.} Compute 
\begin{equation}
y_{k+1}=\arg\min_{x\in\E}\left\{\la\nabla g(y_{k}),x-y_{k}\ra+2L\beta_{d}(y_{k},x)+\varphi(x)\right\}.
\label{eq:A42}
\end{equation}
\noindent\textbf{Step 2.} Set $k:=k+1$ and go to Step 1. 
\end{mdframed}

Algorithm B can be viewed as a particular instance of the NoLips Algorithm in \cite{BBT}. The next lemma gives a lower bound on the functional decrease in terms of the Bregman distance. It corresponds to Lemma 4.1 in \cite{BBT2}. We give its proof here for completeness. 

\begin{lemma}
\label{lem:A7}
Suppose that H1 and H4 hold and let $\left\{y_{k}\right\}_{k\geq 0}$ be a sequence generated by Algorithm B. Then, for all $k\geq 0$, 
\begin{equation}
\tilde{g}(y_{k})-\tilde{g}(y_{k+1})\geq L\beta_{d}(y_{k},y_{k+1}).
\label{eq:A43}
\end{equation} 
\end{lemma}

\begin{proof}
In view of (\ref{eq:A42}), we have
\begin{equation*}
\la\nabla g(y_{k}),y_{k+1}-y_{k}\ra+2L\beta_{d}(y_{k},y_{k+1})+\varphi(y_{k+1})\leq\varphi(y_{k}).
\end{equation*}
Then, 
\begin{equation}
\la\nabla g(y_{k}),y_{k+1}-y_{k}\ra\leq -2L\beta_{d}(y_{k},y_{k+1})-\varphi(y_{k+1})+\varphi(y_{k}).
\label{eq:A44}
\end{equation}
Now, combining H1 and (\ref{eq:A44}), we obtain
\begin{eqnarray*}
g(y_{k+1})&\leq & g(y_{k})+\la\nabla g(y_{k}),y_{k+1}-y_{k}\ra+L\beta_{d}(y_{k},y_{k+1})\\
          &\leq & g(y_{k})-2L\beta_{d}(y_{k},y_{k+1})-\varphi(y_{k+1})+\varphi(y_{k})+L\beta_{d}(y_{k},y_{k+1}).
\end{eqnarray*}
Therefore,
\begin{equation*}
\tilde{g}(y_{k+1})\leq\tilde{g}(y_{k})-L\beta_{d}(y_{k},y_{k+1}),
\end{equation*}
which gives (\ref{eq:A43}).
\end{proof}

The next lemma gives a lower bound on the functional decrease in terms of the norm of a certain subgradient of $\tilde{g}(\,.\,)$.

\begin{lemma}
\label{lem:A8}
Suppose that H1-H4 hold and let $\left\{y_{k}\right\}_{k\geq 0}$ be generated by Algorithm B. Then, for all $k\geq 0$,
\begin{equation}
u(y_{k+1})\equiv \nabla g(y_{k+1})-\nabla g(y_{k})+2L\left[\nabla d(y_{k})-\nabla d(y_{k+1})\right]\in\partial\tilde{g}(y_{k+1}),
\label{eq:A45}
\end{equation}
and
\begin{equation}
\tilde{g}(y_{k})-\tilde{g}(y_{k+1})\geq\dfrac{\sigma_{q}}{qL^{q-1}(3N)^{q}}\|u(y_{k+1})\|_{*}^{q},
\label{eq:A46}
\end{equation}
where $\sigma_{q}$ and $N$ are specified in H2 and H3 (see page 18), respectively.
\end{lemma}

\begin{proof}
By H2 and Lemma \ref{lem:A7}, for all $k$, we have
\begin{equation}
\tilde{g}(y_{k})-\tilde{g}(y_{k+1})\geq L\beta_{d}(y_{k},y_{k+1})\geq\dfrac{L\sigma_{q}}{q}\|y_{k}-y_{k+1}\|^{q}.
\label{eq:A47}
\end{equation}
By the definition of $y_{k+1}$, this point satisfies the first-order optimality condition:
\begin{equation*}
0\in\left\{\nabla g(y_{k})+2L\left[\nabla d(y_{k+1})-\nabla d(y_{k})\right]\right\}+\partial\varphi(y_{k+1}).
\end{equation*}
Since $\text{ri}\left(\dom\varphi\right)\neq\emptyset$, it follows that (\ref{eq:A45}) is true.

On the other hand, by H1, H4 and Proposition 1.1 in \cite{LU}, we have
\begin{equation*}
0\preceq\mu\nabla^{2}d(y)\preceq\nabla^{2}g(y)\preceq L\nabla^{2}d(y),\quad\forall y\in\E.
\end{equation*} 
Consequently, 
\begin{equation}
\|\nabla^{2}g(y)\|\leq L\|\nabla^{2}d(y)\|,\quad\forall y\in\E.
\label{eq:A48}
\end{equation}
Thus, in view of H3 and (\ref{eq:A48}), it follows from the mean value theorem that $\nabla d$ and $\nabla g$ are Lipschitz continuous on $\text{co}\left(\mathcal{L}(y_{0})\right)$ with constants $N$ and $LN$, respectively. From (\ref{eq:A47}), we see that $\left\{\tilde{g}(y_{k})\right\}_{k\geq 0}$ is nonincreasing, and so $\left\{y_{k}\right\}\subset\mathcal{L}(y_{0})$. Therefore, 
\begin{eqnarray*}
\|u(y_{k+1})\|_{*}&\leq &\|\nabla g(y_{k+1})-\nabla g(y_{k})\|_{*}+2L\|\nabla d(y_{k})-\nabla d(y_{k+1})\|_{*}\\
&\leq &\left(LN+2LN\right)\|y_{k}-y_{k+1}\|,
\end{eqnarray*}
that is,
\begin{equation}
\|y_{k}-y_{k+1}\|\geq\dfrac{1}{3LN}\|u(y_{k+1})\|_{*}.
\label{eq:A49}
\end{equation}
Combining (\ref{eq:A46}) and (\ref{eq:A48}), we obtain
\begin{eqnarray*}
\tilde{g}(y_{k})-\tilde{g}(y_{k+1})&\geq &\dfrac{L\sigma_{q}}{q}\dfrac{1}{(3LN)^{q}}\|u(y_{k+1})\|_{*}^{q}\\
& = &\dfrac{\sigma_{q}}{qL^{q-1}(3N)^{q}}\|u(y_{k+1})\|_{*}^{q},
\end{eqnarray*}
which is (\ref{eq:A46}).
\end{proof}

\begin{theorem}
\label{thm:A9}
Suppose that H1-H4 hold and let $\left\{y_{k}\right\}_{k\geq 0}$ be generated by Algorithm B. Then, for all $k\geq 1$, we have
\begin{equation}
\tilde{g}(y_{k})-\tilde{g}(y^{*})\leq\dfrac{\mu\beta_{d}(y_{0},y^{*})}{\left(1+\dfrac{\mu}{2L-\mu}\right)^{k}-1}\leq \dfrac{\left(2L-\mu\right)\beta_{d}(y_{0},y^{*})}{k},
\label{eq:A50}
\end{equation}
where, in case $\mu=0$, the middle expression is defined by the limit as $\mu\to 0^{+}$.
\end{theorem}

\begin{proof}
By H1, for all $k\geq 1$, we have
\begin{equation}
g(y_{k})\leq g(y_{k-1})+\la\nabla g(y_{k-1}),y_{k}-y_{k-1}\ra +L\beta_{d}(y_{k-1},y_{k}).
\label{eq:A51}
\end{equation}
To obtain an upper bound for the inner product in (\ref{eq:A51}), let us apply Lemma A.3 with $h=d$ and 
\begin{equation*}
\phi(x)=\dfrac{1}{2L}\left[\la\nabla g(y_{k-1}),x-y_{k-1}\ra+\varphi(x)\right].
\end{equation*}
In this case, $y^{+}=y_{k}$ and, for $y=y_{k-1}$ we have
\begin{equation*}
\phi(x)+\beta_{d}(y_{k-1},x)\geq\phi(y_{k})+\beta_{d}(y_{k-1},y_{k})+\beta_{d}(y_{k},x),\quad\forall x\in\E,
\end{equation*}
that is,
\begin{eqnarray*}
\la\nabla g(y_{k-1}),x-y_{k-1}\ra+\varphi(x)+2L\beta_{d}(y_{k-1},x)&\geq &\la\nabla g(y_{k-1}),y_{k}-y_{k-1}\ra+\varphi(y_{k})\\
& & +2L\beta_{d}(y_{k-1},y_{k})+2L\beta_{d}(y_{k},x).
\end{eqnarray*}
This gives the upper bound
\begin{eqnarray}
\la\nabla g(y_{k-1}),y_{k}-y_{k-1}\ra &\leq &\la\nabla g(y_{k-1}),x-y_{k-1}\ra+\varphi(x)+2L\beta_{d}(y_{k-1},x)\nonumber\\
& &-\varphi(y_{k})-2L\beta_{d}(y_{k-1},y_{k})-2L\beta_{d}(y_{k},x).
\label{eq:A52}
\end{eqnarray}
Combining (\ref{eq:A51}) and (\ref{eq:A52}), we obtain
\begin{eqnarray*}
g(y_{k})&\leq & g(y_{k-1})+\la\nabla g(y_{k-1}),x-y_{k-1}\ra+\varphi(x)+2L\beta_{d}(y_{k-1},x)\\
        &     & -\varphi(y_{k})-2L\beta_{d}(y_{k-1},y_{k})-2L\beta_{d}(y_{k},x)+L\beta_{d}(y_{k-1},y_{k})
\end{eqnarray*}
\begin{equation}
\tilde{g}(y_{k})\leq g(y_{k-1})+\la\nabla g(y_{k-1}),x-y_{k-1}\ra+\varphi(x)+2L\beta_{d}(y_{k-1},x)-2L\beta_{d}(y_{k},x).
\label{eq:A53}
\end{equation}
Combining (\ref{eq:A53}) and (\ref{eq:4.22}), we get
\begin{eqnarray*}
\tilde{g}(y_{k})&\leq & g(y_{k-1})+g(x)-g(y_{k-1})-\mu\beta_{d}(y_{k-1},x)+\varphi(x)\\
                &     & +2L\beta_{d}(y_{k-1},x)-2L\beta_{d}(y_{k},x)\\
                &=&\tilde{g}(x)+(2L-\mu)\beta_{d}(y_{k-1},x)-2L\beta_{d}(y_{k},x).
\end{eqnarray*}
Substituting $x=y^{*}$, it follows that
\begin{equation*}
\tilde{g}(y_{k})\leq\tilde{g}(y^{*})+(M-\mu)\beta_{d}(y_{k-1},y^{*})-M\beta_{d}(y_{k},y^{*}),
\end{equation*}
where $M=2L$. Then, the rest of the proof follows exactly as in the proof of Theorem A.4 (from inequality (\ref{eq:4.26})).
\end{proof}

\begin{corollary}
\label{cor:A10}
Suppose that H1-H3 hold and let $\left\{y_{k}\right\}_{k\geq 0}$ be a sequence generated by Algorithm B. Additionally, assume that H4 holds with $\mu>0$ and let $u(y_{k})\in\partial\tilde{g}(y_{k})$ be defined in (\ref{eq:A45}) for $k\geq 1$. Given $\epsilon\in (0,1)$, if $\|u(y_{T+1})\|_{*}>\epsilon$, then
\begin{equation*}
T\leq\left[\log_{2}\left(\dfrac{2L}{2L-\mu}\right)\right]^{-1}\left[C+q\right]\log_{2}(\epsilon^{-1}),
\end{equation*}
where 
\begin{equation*}
C=\log_{2}\left(\dfrac{2q(2L)^{q-1}N\mu\beta_{d}(y_{0},y^{*})}{\sigma_{q}}\right).
\end{equation*}
\end{corollary}

\begin{proof}
By Lemma \ref{lem:A8} and Theorem \ref{thm:A9}, it follows as in the proof of Corollary A.6.
\end{proof}

\end{document}